\documentclass{article}

\usepackage{amsfonts, amsmath, amssymb, amsthm, amscd}
\usepackage{ascmac}
\usepackage{verbatim}
\usepackage[dvips]{graphicx}
\usepackage{graphics}
\usepackage[all,2cell]{xypic}
\objectmargin+{1mm}
\labelmargin+{0.8mm}
\SelectTips{cm}{12}

\theoremstyle{plain}  
\newtheorem{thm}{Theorem}[section]

\newtheorem{cor}[thm]{Corollary}
\newtheorem{lem}[thm]{Lemma}
\newtheorem{prop}[thm]{Proposition}

\theoremstyle{definition}

\newtheorem{df}[thm]{Definition}
\newtheorem{ex}[thm]{Example}

\newtheorem{nt}[thm]{Notations}

\newtheorem{rem}[thm]{Remark}

\newtheorem{para}[thm]{}
\newtheorem{lemdf}[thm]{Lemma-Definition}

\theoremstyle{remark}
\newtheorem{ppara}{}[subsection]
\newtheorem*{claim}{Claim}

\DeclareMathOperator{\id}{id}
\DeclareMathOperator{\isoto}{\overset{\scriptstyle{\sim}}{\to}}
\DeclareMathOperator{\rinf}{\rightarrowtail}

\DeclareMathOperator{\rdef}{\twoheadrightarrow}

\DeclareMathOperator{\rinc}{\hookrightarrow}

\newcommand{\onto}[1]{\stackrel{#1}{\to}}

\newcommand{\ssm}{\smallsetminus}
\renewcommand{\coprod}{\sqcup}

\renewcommand{\geqq}{\geq}

\newcommand{\Ker}{\operatorname{Ker}}
\newcommand{\im}{\operatorname{Im}}
\newcommand{\coker}{\operatorname{Coker}}

\newcommand{\Cone}{\operatorname{Cone}}

\newcommand{\Typ}{\operatorname{Typ}}

\newcommand{\adj}{\operatorname{adj}}

\newcommand{\grade}{\operatorname{grade}}

\newcommand{\rank}{\operatorname{rank}}

\newcommand{\Supp}{\operatorname{Supp}}

\DeclareMathOperator{\Ob}{Ob}

\DeclareMathOperator{\op}{op}

\DeclareMathOperator{\Hom}{Hom}

\DeclareMathOperator{\Tot}{Tot}

\DeclareMathOperator{\Homo}{H}

\newcommand{\bbZ}{\operatorname{\mathbb{Z}}}

\DeclareMathOperator{\fA}{\mathfrak{A}}
\DeclareMathOperator{\fa}{\mathfrak{a}}

\DeclareMathOperator{\fb}{\mathfrak{b}}

\DeclareMathOperator{\fc}{\mathfrak{c}}
\DeclareMathOperator{\fd}{\mathfrak{d}}

\DeclareMathOperator{\frakj}{\mathfrak{j}}
\DeclareMathOperator{\ff}{\mathfrak{f}}

\DeclareMathOperator{\fh}{\mathfrak{h}}

\DeclareMathOperator{\fP}{\mathfrak{P}}

\DeclareMathOperator{\fU}{\mathfrak{U}}

\DeclareMathOperator{\fx}{\mathfrak{x}}

\DeclareMathOperator{\fz}{\mathfrak{z}}

\DeclareMathOperator{\cA}{\mathcal{A}}
\DeclareMathOperator{\cB}{\mathcal{B}}
\DeclareMathOperator{\cC}{\mathcal{C}}
\DeclareMathOperator{\calD}{\mathcal{D}}

\DeclareMathOperator{\cP}{\mathcal{P}}

\DeclareMathOperator{\cX}{\mathcal{X}}
\DeclareMathOperator{\cY}{\mathcal{Y}}

\newcommand{\adm}{\operatorname{adm}}
\newcommand{\DP}{\operatorname{\mathcal{DP}}}
\newcommand{\Fib}{\operatorname{Fib}}

\newcommand{\CoCub}{\operatorname{\bf CoCub}}
\newcommand{\CoDCub}{\operatorname{\bf CoDCub}}
\newcommand{\Cub}{\operatorname{\bf Cub}}
\newcommand{\DCub}{\operatorname{\bf DCub}}
\newcommand{\Pat}{\operatorname{\bf Pat}}
\newcommand{\Ide}{\operatorname{\bf Ideal}}

\newcommand{\Suc}{\operatorname{Suc}}
\newcommand{\Pre}{\operatorname{Pre}}

\UseAllTwocells
\def\sn{\smallskip\noindent}

\newcommand{\cf}{\textrm{cf.}\;}

\title{What makes a multi-complex exact?}
\author{Satoshi Mochizuki \and Seidai Yasuda}
\date{}

\begin{document}

\maketitle

\begin{abstract}
In this paper, 
we give a sufficient condition which makes the total complex of 
a cube exact. 
This can be regarded as 
a variant of the Buchsbaum-Eisenbud theorem 
which gives a characterization of what makes a complex of 
finitely generated free modules exact in terms of the grade 
of the Fitting ideals of boundary maps of the complex.
\end{abstract}

\section*{Introduction}
\label{sec:Intro} 

\sn
In the celebrated paper \cite{BE73}, 
Buchsbaum and Eisenbud gave 
a necessary and sufficient condition for 
a complex $x$ of finitely generated free modules to be 
a resolution of the $0$-th homology group $\Homo_0 x$ of $x$ 
(we call such a complex {\bf $0$-spherical}) 
in terms of the {\it grade} of the {\it Fitting ideals} of 
boundary morphisms of the complex $x$. 
(See Theorem~\ref{thm:BEthm} and also \cite[1.4.2]{BJ93}.) 
The main goal of this paper 
is to give a variant of this theorem for certain multi-complexes. 

\sn
To state our problem precisely, 
we need to consider what is a natural generalization the notion 
of $0$-spherical complexes for multi-complexes. 
In this paper we restrict ourselves to 
a special class of multi-complexes which we call {\it cubes}. 
We take the position that admisible cubes introduced in \cite{Kos} 
are the counterparts of $0$-spherical complexes. 
In the papers \cite{Kos}, \cite{MS13} and \cite{Moc13b}, 
we studied resolution of modules by admissible cubes of free modules 
and applied them to calculate the derived categories and the $K$-theory of 
modules. 

\sn
To explain our result precisely, 
we need to introduce some notation. 
Let $S$ be a finite set. 
We denote the cardinality of the set $S$ by $\# S$. 
We write $\cP(S)$ for the {\bf power set of $S$}. 
We regard $\cP(S)$ as a category whose 
objects are subsets of $S$ and whose 
morphisms are inclusions. 
An {\bf $S$-cube} in a category $\calD$ is 
a contravariant functor 
from $\cP(S)$ to $\calD$. 
Let $x$ be an $S$-cube in $\calD$. 
For any $T\in\cP(S)$, 
we denote $x(T)$ by $x_T$ and call it the 
{\bf vertex of $x$} ({\bf at $T$}). 
For any $k \in T$, 
we also write $d^{x,k}_T$ or 
shortly $d^k_T$ 
for $x(T\ssm\{k\} \rinc T)$ 
and call it 
the ({\bf $k$-}){\bf boundary morphism of $x$} ({\bf at $T$}). 
(See Definition~\ref{df:cubes} and 
Example~\ref{ex:various type cubes}.) 
Let $F^k$ and $B^k$ be the order preserving maps from 
$\cP(S\ssm\{k\})$ to $\cP(S)$ defined 
by sending a subset $U$ of $S\ssm\{k\}$ to $U$ and $U\cup\{k\}$ 
respectively. 
For any $S$-cube $x$ 
in a category $\calD$ 
and any $k\in S$, 
the compositions 
$xB^k$ and $xF^k$ are called 
the {\bf backside $k$-face of $x$} and 
the {\bf frontside $k$-face of $x$} respectively. 
By a {\bf face} of $x$, 
we mean any backside or frontside $k$-face of $x$. 
(See Example~\ref{ex:Faces of cubes}.) 
Let $S$ be a non-empty finite set and 
$x$ an $S$-cube in an additive category $\cB$. 
If we regard $x$ as a multi-complex where $x_{\emptyset}$ is in 
degree $(0,\cdots,0)$, we will take its total complex $\Tot x$. 
(See Notations~\ref{nt:tot complex}.) 

\begin{ex}
[\bf See Notations~\ref{nt:A-lincat} and Example~\ref{ex:abs typ cubes}]
\label{ex:A-lincat, abs typ cubes}
Let $A$ be a commutative ring with unit and 
$\cC$ an abelian category enriched over the category of $A$-modules. 
Namely for any pair of objects $x$ and $y$ in $\cC$, 
the set of morphisms from $x$ to $y$, $\Hom_{\cC}(x,y)$ has 
a structure of $A$-module and 
the composition of morphisms 
$\Hom_{\cC}(x,y)\times \Hom_{\cC}(y,z) \to \Hom_{\cC}(x,z)$ 
is an $A$-bilinear homomorphism for any objects $x$, $y$ and $z$ in $\cC$. 
In particular, it induces a homomorphism of $A$-modules 
$\Hom_{\cC}(x,y)\otimes_A \Hom_{\cC}(y,z) \to \Hom_{\cC}(x,z) $. 
For any object $x$ in $\cC$ and any element $a$ in $A$, 
we write $a_x$ for the morphism $a\id_x:x\to x$. 
A typical example of $\cC$ is the category of $A$-modules. 
Let $\ff_S=\{f_s\}_{s\in S}$ a family of elements in $A$ and $x$ an object in $\cC$. 
We define the $S$-cube $\Typ(\ff_S;x)$ 
in $\cC$, 
called the {\bf typical cubes associated with 
$\ff_S$ and $x$}, as follows. 
For any $T\in\cP(S)$ and any element $t$ in $T$, 
we put $\Typ(\ff_S;x)_T=x$ and $d_T^{t,\Typ(\ff_S;x)}:={(f_t)}_x$. 
In particular, if $\cC$ is the category of $A$-modules, then 
$\Tot \Typ(\ff_S;A)$ is isomorphic to the usual {\it Koszul complex} associated with 
the family of elements $\ff_S$. 
\end{ex}

\sn
In particular, we study a specific class of cubes in an abelian category 
which is a categorical variant of the notion about regular sequences 
in commutative rings. 
Let $S$ be a finite set and $\cA$ an abelian category. 
Let us fix an $S$-cube $x$ 
in $\cA$. 
For each $k \in S$, 
the {\bf $k$-direction $0$-th homology} of $x$ 
is the $S\ssm\{k\}$-cube $\Homo_0^k(x)$ in $\cA$ 
defined by $\Homo_0^k(x)_T:=\coker d_{T\cup\{k\}}^k$. 
When $\# S=1$, 
we say that $x$ is {\bf admissible} if 
its unique boundary morphism is a monomorphism. 
For $\# S>1$, 
we define the notion of an admissible cube inductively 
by saying that 
$x$ is {\bf admissible} if 
its boundary morphisms are monomorphisms and if 
for every $k$ in $S$, 
$\Homo^k_0(x)$ is admissible. 
(See Definition~\ref{para:admcube}.) 

\sn 
The relationship between admissibility of cubes and the classical notion of 
regular sequences are summed up with the following way. 
For any elements $f_1,\cdots,f_q$ in $A$ and an object $x$ in $\cC$, 
we simply write $x/(f_1,\cdots,f_q)$ for $x/({(f_1)}_x,\cdots,{(f_q)}_x)$. 
Let us fix an object $x$ in $\cC$. 
A sequence of elements $f_1,\cdots,f_q$ in $A$ is 
an {\bf $x$-regular sequence} 
if every $f_i$ is a non-unit in $A$, if 
${(f_1)}_x$ is a monomorphism in $\cC$ and if 
${(f_{i+1})}_{x/(f_1,\cdots,f_i)}$ is a monomorphism for any $1\leq i\leq q-1$. 
A finite family $\{f_s\}_{s\in S}$ of elements in $A$ is an {\bf $x$-sequence} 
if $\{f_s\}_{s\in S}$ forms an $x$-regular sequence 
with respect to every ordering of the members of $\{f_s\}_{s\in S}$. 
Let $\ff_S=\{f_s\}_{s\in S}$ be a family of elements in $A$. 
Then the $S$-cube $\Typ(\ff_S;x)$ is admissible 
if and only if 
the family $\ff_S$ is an $x$-sequence. 
(See Notations~\ref{nt:regular sequences} and Lemma~\ref{lem:char of x-seq}.) 
There are several characterizations of 
admissibility of cubes in an abelian category. 
In particular, 
the admissibility of a cube $x$ gives 
a sufficient condition 
that the cube $x$ is a resolution of 
the $0$-th homology group 
$\Homo_0\Tot x$ of $\Tot x$. 
(See Theorem~\ref{thm:char of adm} for details.)

\sn
In our main theorem, 
we give a sufficient condition of admissibility of cubes. 
To state the main theorem, 
we introduce some terminology about a categorical variant of 
adjugates of matrices. 
The existence of regular adjugates of cubes 
implies the condition about the grade of the Fitting ideals of boundary morphisms 
of complexes in the Buchsbaum-Eisenbud theorem in \cite{BE73}. 
(See Proposition~\ref{prop:relation with BE}.) 
Let $S$ be a finite set and $\cC$ be an abelian category as 
in Example~\ref{ex:A-lincat, abs typ cubes}. 
An {\bf adjugate of an $S$-cube} $x$ in $\cC$ is a pair $(\fa,\fd^{\ast})$ 
consisting of a family of elements $\fa=\{a_s\}_{s\in S}$ in $A$ 
and a family of morphisms 
$\fd^{\ast}=\{d_T^{t \ast}\colon x_{T\ssm\{t\}} \to x_T\}_{T\in\cP(S),t\in T}$ 
in $\cC$ which satisfies the following two conditions.\\
$\mathrm{(i)}$ 
We have the equalities 
$d_T^td_T^{t \ast}={(a_t)}_{x_{T\ssm\{t\}}}$ and 
$d_T^{t \ast}d_T^t={(a_t)}_{x_T}$ 
for any $T\in\cP(S)$ and $t\in T$.\\
$\mathrm{(ii)}$ 
For any $T\in\cP(S)$ and any two distinct elements $a$ and $b\in T$, 
we have the equality 
$d_T^bd_T^{a \ast}=d^{a \ast}_{T\ssm\{b\}}d^b_{T\ssm\{a\}}$. 
An adjugate of an $S$-cube $(\fa,\fd^{\ast})$ is {\bf regular} 
if $\fa$ forms $x_T$-sequence for any $T\in \cP(S)$. 
(See Definition~\ref{df:adjugate of cubes}.) 
Example~\ref{ex:adj cube} shows 
how to relate the notion about adjugates of cubes and 
the classical notion about adjugates of matrices. 
The following theorem is the main theorem in this paper. 

\begin{thm}[\bf A part of Theorem~\ref{thm:main thm}]
\label{introthm}
Let $\cC$ be an abelian category as 
in Example~\ref{ex:A-lincat, abs typ cubes} 
and $x$ be an $S$-cube in $\cC$. 
If  $x$ admits a regular adjugate, 
then $x$ is admissible.
\end{thm}

\sn
$2$-dimensional cubes whose boundary morphisms are monomorphisms 
are admissible if and only if 
it is Cartesian. (See Example~\ref{ex:adm squ}.) 
However a similar statement is no longer valid for higher dimensional cubes. 
Cartesian cubes 
(we call them {\bf fibered cubes} in Definition~\ref{df:fibered S,P-cubes}) 
are not necessarily admissible. 
Theorem~\ref{introthm} provides a useful criterion 
for a Cartesian cube to be admissible. 
As application we can derive, as an immediate consequence of the main theorem, 
a property stated in Corollary~\ref{cor:app 1} of regular sequences, 
for which one of the author gave a proof in \cite[Lemma~4.2]{Kos} 
by a more straightforward but complicated method.

\sn
Theorem~\ref{thm:main thm} is a consequence of 
Theorem~\ref{thm:dct} which is a purely general theorem in category theory. 
For any natural number $n$, 
let $[n]^S$ be the partially ordered set of 
maps from $S$ to the totally ordered set of integers $k$ satisfying $0\leq k\leq n$. 
Let $\mathbf{2}:[1]^S \to [2]^S$ and $e_T:[1]^S \to [2]^S$ 
be order preserving maps defined by sending a map $f:S \to [1]$ 
to $2f:S \to [2]$ and $f+\chi_T:S \to [2]$ 
respectively 
where $\chi_T$ is the characteristic function of $T$ on $S$. 
(Compare Example~\ref{ex:power set}.) 
A {\bf double $S$-cube} $x$ in a category $\calD$ is 
a contravariant functor from $[2]^S$ to $\calD$. 
(See Example~\ref{ex:various type cubes}.)

\begin{thm}[\bf Double cube theorem]
\label{thm:dct}
Let $x$ be a double $S$-cube in an abelian category $\cA$. 
We assume that the following conditions hold.\\
$\mathrm{(1)}$ 
The $S$-cube $x\mathbf{2}$ is admissible.\\
$\mathrm{(2)}$ 
For any ordering pair 
$\frakj<\frakj'$ in $[2]^S$, 
$x(\frakj<\frakj')$ is a monomorphism in $\cA$.\\
$\mathrm{(3)}$ 
If $\# S\geq 3$, 
all faces of the $S$-cube 
$xe_T$ are admissible 
for any proper subset $T$ of $S$.\\
Then the $S$-cube $xe_S$ is also an admissible $S$-cube.
\end{thm}

\sn
The proof of Theorem~\ref{thm:dct} 
will be given at \ref{para:proof dct}. 
We explain the structure of this paper. 
In section~\ref{sec:adm seq in mod lat}, 
we introduce and study the notion about 
universally admissible families 
in a lattice which is 
a lattice theoretic variant of regular sequences 
in commutative ring theory. 
In section~\ref{sec:cubes}, 
we introduce and study the notions of (co)cubes and fibered cubes. 
In section~\ref{sec:adm cubes}, 
we review and establish the foundation of 
admissible cubes in an abelian category 
from \cite{Kos}. 
In section~\ref{sec:double cubes}, 
we develop an abstract version of 
the main theorem. 
In section~\ref{sec:reg adj cubes}, 
we state and prove the main theorem. 
The standard results in this paper 
will be frequently utilized in the authors' subsequent works 
about studying the weight of Adams operations on 
topological filtrations of $K$-theory of 
commutative regular local rings. 

\section*{Conventions.}
\subsection{General assumptions}
Throughout this paper, 
we use the letters $A$, $\calD$, 
$\cA$, $S$ and $P$ to denote a commutative ring with unit, 
a category, an abelian category, a set 
and a partially ordered set respectively.

\subsection{Partially ordered sets}
\begin{ppara}
For two elements $a$, $b$ in a partially ordered set $P$, 
we write $[a,b]$ for 
the set of all elements $u$ in $P$ satisfying $a\leq u\leq b$. 
We regard $[a,b]$ as a partially ordered subset of $P$ if $a\leq b$ 
and $[a,b]=\emptyset$ if otherwise. 
We often use this notation 
when $P=\bbZ$ is the partially ordered set of integers.
\end{ppara}
\begin{ppara}
For a non-negative integer $n$ and a positive integer $m$, 
we denote $[0,n]$ and $[1,m]$ by $[n]$ and $(m]$ respectively.
\end{ppara}
\begin{ppara}
The {\bf trivial ordering} $\leq$ on a set $S$ is defined by 
$x\leq y$ if and only if $x=y$.
\end{ppara}
\begin{ppara}
An element $x$ in a partially ordered set $P$ is said to 
be {\bf maximal} (resp. {\bf minimal}) if 
for any element $a$ in $P$, 
the inequality $x\leq a$ (resp. $a\leq x$) implies 
the equality $x=a$. 
An element $x$ in a partially ordered set $P$ is 
{\bf maximum} (resp. {\bf minimum}) if 
the inequality $a\leq x$ (resp. $x\leq a$) 
holds for any elements $a$ in $P$.
\end{ppara}

\begin{ppara}
For any set $S$, 
we write $\cP(S)$ for its {\bf power set}. 
Namely $\cP(S)$ is the set of all subsets of $S$. 
We consider $\cP(S)$ 
to be a partially ordered set under inclusion. 
\end{ppara}

\subsection{Category theory}
\begin{ppara}
We say a category $\cX$ is {\bf locally small} (resp. {\bf small}) if 
for any objects $x$ and $y$, $\Hom_{\cX}(x,y)$ forms a set 
(resp. if $\cX$ is locally small and $\Ob\cX$ forms a set).
\end{ppara}
\begin{ppara}
For two categories $\cX$ and $\cY$, 
we denote the (large) category of functors from $\cX$ to $\cY$ 
by $\cY^{\cX}$. 
Here the morphisms between functors from 
$\cX$ to $\cY$ are just natural transformations.
\end{ppara}
\begin{ppara}
We regard a partially ordered set $P$ as a category in a natural way. 
Namely, $P$ is a small category whose 
set of objects is $P$ and for any elements $x$ and $y$ in $P$, 
$\Hom_P(x,y)$ is the singleton $\{(x,y)\}$ if $x\leq y$ and is 
the empty set $\emptyset$ if otherwise. 
In particular, 
we regard any set $S$ as a category by the trivial ordering on $S$.
\end{ppara}

\subsection{Chain complexes}
For a chain complex, we use the homological notation. 
Namely a boundary morphisms are of degree $-1$.

\section{Universally admissible families in lattices}
\label{sec:adm seq in mod lat}

In this section, we study the notion about ({\it universally}) 
{\it admissible families} in lattices. 
Let us start by recalling some basic concepts about lattices. 

\begin{df}[\bf Lattice]
\label{df:lattice}
A {\bf lattice} $L$ is a partially ordered set such that for any elements 
$a$ and $b$ in $L$, 
their supremum $a\vee b$ and 
their infimum $a\wedge b$ exist. 
We call $a\vee b$ (resp. $a\wedge b$ the {\bf join} (resp. the {\bf meet}) of $a$ and $b$. 
\end{df}

\begin{nt}
\label{nt:x^wedge}
Let $S$ be a non-empty finite set and 
$\fx=\{x_s\}_{s\in S}$ a family of elements in a lattice $L$.\\
$\mathrm{(1)}$ 
For any subset $T$ of $S$, 
we denote the subfamily $\{x_t\}_{t\in T}$ by $\fx_T$.\\
$\mathrm{(2)}$ 
We write $\fx^{\vee S}$ or $\displaystyle{\underset{s\in S}{\bigvee} x_s}$ 
(resp. $\fx^{\wedge S}$ or $\displaystyle{\underset{s\in S}{\bigwedge} x_s}$) 
for 
$\sup\{x_s;s\in S \}$ (resp. $\inf\{x_s;s\in S \}$) and 
call $\fx^{\vee S}$ (resp. $\fx^{\wedge S}$) the {\bf join} (resp. {\bf meet}) of a family $\fx$. 
For any non-empty subset $T$ of $S$, 
we write $\fx^{\vee T}$ and $\fx^{\wedge T}$ for ${(\fx_T)}^{\vee T}$ 
and ${(\fx_T)}^{\wedge T}$ respectively. 
If the lattice $L$ has the maximum element $1$, 
we use the notation $\fx^{\wedge S}$ or $\displaystyle{\underset{s\in S}{\bigwedge} x_s}$ 
for $S=\emptyset$, 
which stands for the element $1$.\\
$\mathrm{(3)}$ 
We write $\fx\vee y$ (resp. $\fx\wedge y$) 
for a family $\{x_s\vee y\}_{s\in S}$ (resp. $\{x_s\wedge y\}_{s\in S}$). 
We have the following inequalities. 
\begin{equation}
\label{equ:pardist vee}
{(\fx\wedge y)}^{\vee S}\leq\fx^{\vee S}\wedge y.
\end{equation}
\begin{equation}
\label{equ:pardist wedge}
{(\fx\vee y)}^{\wedge S}\geq\fx^{\wedge S}\vee y.
\end{equation} 
\end{nt}

\begin{df}[\bf Ideals]
\label{df:ideals}
A subset $I$ of a partially ordered set $P$ is an {\bf ideal} if 
for any pair of elements $x\leq y$ of $P$, 
$x\in I$ implies $y\in I$. 
We write $\Ide(P)$ for the set of all ideals in $P$. 
\end{df}

\begin{rem}[\bf Semi-modular law]
\label{rem:semi-mod law}
Let $L$ be a lattice and $a$, $b$ and $c$ elements in $L$ such that $a\leq c$. 
Then we have the following inequality called the {\bf semi-modular law}. 
\begin{equation}
\label{equ:semi-mod law}
a\vee (b\wedge c) \leq (a\vee b)\wedge c.
\end{equation}
\end{rem}

\begin{lemdf}[\bf Modular lattice]
\label{lemdf:modular lattice}
A lattice $L$ is {\bf modular} 
if the following equivalent conditions hold.\\
$\mathrm{(1)}$ 
For any elements $a$, $b$ and $c$ in $L$ such that $a\leq c$, 
the following equality called the {\bf modular law} holds.
\begin{equation}
\label{equ:modular law}
a\vee(b\wedge c)=(a\vee b)\wedge c.
\end{equation}
The modular law is equivalent to the inequality 
\begin{equation}
\label{equ:modular law 2}
a\vee(b\wedge c)\geq(a\vee b)\wedge c
\end{equation}
by the inequality $\mathrm{(\ref{equ:semi-mod law})}$ 
in Remark~\ref{rem:semi-mod law}.\\
$\mathrm{(2)}$ 
For any elements $a$, $b$ and $c$ in $L$ such that $a\leq b$, 
the equalities $a\vee c=b\vee c$ and $a\wedge c=b\wedge c$ imply 
the equality $a=b$. 
\end{lemdf}

\begin{proof}[\bf Proof]
First we assume that condition $\mathrm{(1)}$ holds. 
Then for any elements $a$, $b$ and $c$ in $L$ such that 
$a\leq b$, $a\vee c=b\vee c$ and $a\wedge c=b\wedge c$, 
we have the equalities 
$$a=a\vee(a\wedge c)=a\vee(b\wedge c)=(a\vee c)\wedge b=(b\vee c)\wedge b=b.$$
Next we assume that condition $\mathrm{(2)}$ holds. 
Then for any elements $a$, $b$ and $c$ in $L$ such that $a\leq c$, 
we put $x=a\vee(b\wedge c)$ and $y=(a\vee b)\wedge c$. 
Then we have $x\leq y$, $x\vee b=y\vee b$ and $x\wedge b=y\wedge b$. 
Hence we have the equalities 
$$a\vee(b\wedge c)=x=y=(a\vee b)\wedge c.$$
\end{proof}

\begin{ex}[\bf Well-powered abelian category]
\label{ex:well-powered abelian category}
An abelian category $\cA$ is {\bf well-powered} if 
for any object $x$ in $\cA$, 
the isomorphism class of subobjects of $x$ 
which is written by $\cP(x)$ 
forms a set. 
For example, 
it is well-known that if $\cA$ is the category of $A$-modules, 
then $\cA$ is well-powered. 
We claim that 
for any object $x$ in a well-powered abelian category $\cA$, 
the set $\cP(x)$ is a modular lattice 
with respect to the ordering given by the inclusion. 
For any abelian category $\cA$, for any object $x$ in $\cA$ 
and for any family of subobjects 
$\fx=\{x_s \rinf x\}_{s\in S}$ 
indexed by a set $S$, 
we write $\cP(\fx)$ 
for the sublattice of $\cP(x)$ generated by $\fx$. 
Then $\cP(\fx)$ is a modular lattice. 
\end{ex}

\begin{proof}[\bf Proof of the claim in 
Example~\ref{ex:well-powered abelian category}]
In general, for any subobjects $a\subset b$ and $c\subset d$ of $x$, 
we have the short exact sequence 
\begin{equation}
\label{equ:can short exact seq}
(b\wedge d)/(a\wedge c) \rinf b/a \oplus d/c \rdef (b\vee d)/ (a\vee c). 
\end{equation}
By putting $c=d$ in the short exact sequence 
$\mathrm{(\ref{equ:can short exact seq})}$ above, 
it turns out that the equalities 
$a\wedge c=b \wedge c$ and $a\vee c=b\vee c$ imply the equality $a=b$. 
Hence $\cP(x)$ is modular by 
Lemma-Definition~\ref{lemdf:modular lattice} $\mathrm{(2)}$.
\end{proof}

\begin{df}[\bf Distributive, regular and (universally) admissible sequences]
\label{df:dis, reg, adm seq}
Let $L$ be a lattice, $r$ a positive integer and 
$\fx=\{x_s\}_{s\in S}$ a non-empty family of elements in 
$L$ and $y$ an element in $L$.\\
$\mathrm{(1)}$ 
We say that a pair $(\fx,y)$ 
is {\bf distributive} if 
we have an equality 
$\fx^{\vee S}\wedge y={(\fx \wedge y)}^{\vee S}$. 
It is equivalent to the condition that 
$\fx^{\vee S}\wedge y\leq{(\fx \wedge y)}^{\vee S}$ 
by the inequality $\mathrm{(\ref{equ:pardist vee})}$ 
in Notation~\ref{nt:x^wedge}.\\
$\mathrm{(2)}$ 
We say that the family $\fx$ is {\bf strictly distributive} 
(resp. {\bf admissible}) 
if $\# S\leq 1$ or if $\#S\geq 2$ and 
if for any element $t$ in $S$ 
(resp. for any non-empty proper subset $T$ of $S$ and 
for any element $t$ in $S\ssm T$) a pair $(\fx_{S\ssm\{t\}},x_t)$ (resp. $(\fx_T,x_t)$) 
is distributive.\\
$\mathrm{(3)}$ 
We say that a sequence $z_1,\cdots,z_r$ of elements in $L$ is {\bf regular} if 
$r=1$ or if $r\geq 2$ and 
for any integer $i\in [2,r]$, 
a pair $(\{z_j\}_{1\leq j\leq i-1},z_i)$ is distributive.\\
$\mathrm{(4)}$ 
We say that the family $\fx$ is 
{\bf universally admissible} if 
for any two non-empty subsets $U$ and $V$ of $S$ with 
$U\cap V=\emptyset$, 
the pair $(\fx_U,\fx^{\wedge V})$ is distributive, 
or equivalently, 
if for any disjoint decomposition $S=U\coprod V$ 
such that $U\neq \emptyset$, 
a family $\fx_U\wedge \fx^{\wedge V}$ is admissible. 
\end{df}

\begin{rem}[\bf Universally admissible sequences]
\label{rem:univ adm seq}
Let $S$ be a finite set, $T$ a non-empty subset of $S$ and 
$\fx=\{x_s\}_{s\in S}$ a family of elements in a lattice $L$. 
Then\\
$\mathrm{(1)}$ 
If a family $\fx$ is admissible (resp. universally admissible), 
then a family $\fx_T$ is also admissible (resp. universally admissible).\\
$\mathrm{(2)}$ 
If $\# S\leq 2$, then a family $\fx$ is universally admissible.\\
$\mathrm{(3)}$ 
If $\# S\geq 3$, 
a family $\fx$ is universally admissible if and only if 
$\fx$ satisfies the following two conditions.\\
$\mathrm{(i)}$ 
$\fx$ is admissible.\\
$\mathrm{(ii)}$ 
$\fx_{S\ssm\{s\}}\wedge x_s$ is 
universally admissible for any $s\in S$.\\ 
$\mathrm{(4)}$ 
In particular if $\# S=3$, then 
a family $\fx$ is universally admissible if and only if 
$\fx$ is admissible. 
\end{rem}

\begin{proof}[\bf Proof]
Assertion $\mathrm{(1)}$ for the admissible case is trivial. 
Let us assume that a family $\fx$ is universally admissible. 
For any disjoint decomposition $T=U\coprod V$ such that $\# U\geq 3$, 
a family $\fx_{(S\ssm T)\coprod U}\wedge \fx^{\wedge V}$ is 
admissible by the assumption. 
Therefore a family $\fx_U\wedge \fx^{\wedge V}$ is 
also admissible by the assertion for the admissible case. 
Hence a family $\fx_T$ is universally admissible. 
Next we prove assertion $\mathrm{(3)}$. 
Let us assume that a family $\fx$ satisfies 
conditions $\mathrm{(i)}$ and $\mathrm{(ii)}$ 
and let us fix a pair of disjoint subsets $U$ and $V$ of $S$ 
such that $S=U\coprod V$ and 
$\# U\geq 3$. 
If $V=\emptyset$, then $\fx_U\wedge \fx^{\wedge V}=\fx$ 
is admissible by condition $\mathrm{(i)}$. 
If there exists an element $s$ in $V$, 
then 
$\fx_U\wedge \fx^{\wedge V}=
{(\fx_{S\ssm\{s\}}\wedge x_s)}_{U}\wedge 
{(\fx_{S\ssm\{s\}}\wedge x_s)}^{\wedge V\ssm\{s\}}$ 
is admissible by condition $\mathrm{(ii)}$. 
Hence $\fx$ is universally admissible. 
Assertion of the other direction is trivial. 
Assertions $\mathrm{(2)}$ and $\mathrm{(4)}$ are easy. 
\end{proof}

\begin{ex}[\bf Regular sequences]
\label{ex:reg seq}
Let $A$ be a commutative ring with unit, $M$ an $A$-module, 
$r$ an integer such that $r\geq 3$ and 
$f_1,\cdots,f_r$ a sequence of non-unit elements in $A$. 
Let us recall that we say a sequence $f_1,\cdots,f_r$ is {\bf $M$-regular} if 
the multiplication by $f_1$, $M\to M$ is injective and 
for any $i\in (r-1]$, 
the multiplication by $f_{i+1}$, 
$M/(f_1,\cdots,f_i)M \to M/(f_1,\cdots,f_i)M$ is injective. 
Now assume that a sequence $f_i,f_j$ is a $M$-regular sequence 
for any $1\leq i,\ j\leq r$ with $i\neq j$. 
Then a sequence $f_1M,\cdots,f_rM$ in $\cP(M)$ is a regular sequence 
if and only if the sequence $f_1,\cdots,f_r$ is $M$-regular. 
(See also Lemma~\ref{lem:char of x-seq}.)
\end{ex}

\begin{ex}[\bf Distributive lattices]
\label{ex:distributive lattice}
We say that a lattice $L$ is {\bf distributive} 
if for any finite subset $\fx=\{x_s\}_{s\in S}$ of $L$ 
indexed by a non-empty finite subset $S$ with $\#S \geq 2$, 
a pair $(\fx_{S\ssm\{s\}},x_s)$ is distributive for any $s\in S$. 
For any non-empty finite subset $\fx=\{x_s\}_{s\in S}$ of $L$, 
we consider the following three assertions.\\
$\mathrm{(1)}$ 
The sublattice of $L$ generated by $\fx$ is distributive.\\
$\mathrm{(2)}$ 
The map $\Ide(\cP(S)) \to L$ 
defined by sending an ideal $I$ to 
an element 
$\displaystyle{\underset{V\in I}{\vee}
\fx^{\wedge V}}$ in $L$ 
preserves the meet operation.\\
$\mathrm{(3)}$ 
The set $\{\fx^{\wedge V};V\subset S\}$ is admissible.\\
Then assertions $\mathrm{(1)}$ and $\mathrm{(2)}$ are equivalent and 
assertion $\mathrm{(2)}$ implies assertion $\mathrm{(3)}$. 
Moreover if $\fx$ satisfies assertion $\mathrm{(3)}$, 
then $\fx$ is universally admissible. 
\end{ex}

\begin{prop}
\label{prop:adm seq lem}
Let $L$ be a lattice, 
$S$ a non-empty finite set, 
$\fx=\{x_i\}_{i\in S}$ a family of elements in $L$ 
and $y$ an element in $L$. 
Assume that the following two conditions hold.\\
$\mathrm{(1)}$ 
The family $\fx$ is admissible {\rm (}resp. universally admissible{\rm )}.\\
$\mathrm{(2)}$ 
If $\# S\geq 2$, 
then for any element $s \in S$ and 
any non-empty subset $U$ of $S\ssm\{s\}$, 
a pair $(\fx_{U}\wedge x_s,y)$ 
is distributive. {\rm (}Resp. 
For any pair of non-empty disjoint subsets $U$ and $V$ of 
$S$ 
such that $\# V\geq 2$ 
and any element $v\in V$, 
a pair $(\fx_U\wedge \fx^{\wedge V},\fx^{\wedge V\ssm\{v\}}\wedge y)$ 
is distributive{\rm )}.\\
Then a family $\fx\wedge y$ 
is also admissible {\rm (}resp. universally admissible{\rm )}. 
\end{prop}

\begin{proof}[\bf Proof] 
We first prove the assertion for the admissible case. 
What we need to prove is that 
the family $\fx_U\wedge y$ is strictly distributive 
for any non-empty subset $U$ of $S$. 
We shall assume that $k:=\# U\geq 3$ and 
we put $U=\{i_1,\cdots,i_k\}$. 
Then without loss of generality, 
we just need to check the following (in)equalities. 
$$\underset{j=1}{\overset{k-1}{\bigvee}}(x_{i_j}\wedge x_{i_k}\wedge y) 
\underset{\textbf{I}}{=} 
\left \{\underset{j=1}{\overset{k-1}{\bigvee}} (x_{i_j}\wedge x_{i_k} )\right \} \wedge y
\underset{\textbf{II}}{=} 
\left ( \underset{j=1}{\overset{k-1}{\bigvee}} x_{i_j} \right )\wedge (x_{i_k}\wedge y) \geq 
\left \{\underset{j=1}{\overset{k-1}{\bigvee}} (x_{i_j}\wedge y) \right \} \wedge (x_{i_k}\wedge y)$$
where the equality $\textbf{I}$ follows 
from assumption $\mathrm{(2)}$ 
and the equality $\textbf{II}$ follows 
from assumption $\mathrm{(1)}$. 
Next we prove the assertion for the universally admissible case. 
What we need to prove is that 
for any disjoint decomposition $U\coprod V=S$ such that 
$U\neq \emptyset$, 
$\fx_U\wedge \fx^{\wedge V}\wedge y$ is admissible. 
To prove the assertion above, 
we apply this Proposition~\ref{prop:adm seq lem} 
for the admissible case 
to the family $\fx_U\wedge \fx^{\wedge V}$ and 
the element $\fx^{\wedge V}\wedge y$. 
What we need to check is the following two conditions.\\
$\mathrm{(i)}$ 
The family $\fx_U\wedge \fx^{\wedge V}$ 
is admissible.\\
$\mathrm{(ii)}$ 
For any $u \in U$ and any non-empty subset $W$ of $U\ssm\{u\}$, 
a pair 
$(\fx_W\wedge\fx^{\wedge V\coprod \{u\}},\fx^{\wedge V}\wedge y)$ 
is distributive.\\
Condition $\mathrm{(i)}$ is a consequence of 
assumption $\mathrm{(1)}$ 
and condition $\mathrm{(ii)}$ 
is just assumption $\mathrm{(2)}$. 
Hence we get the desired result.
\end{proof}

\begin{cor}
\label{cor:adm seq cor}
Let $L$ be a lattice, 
$S$ a non-empty finite set 
and $\fa=\{a_s\}_{s\in S}$, $\fb=\{b_s\}_{s \in S}$
families of elements 
indexed by $S$ 
in $L$ 
such that $a_s\geq b_s$ for any $s\in S$. 
Assume that the following two conditions hold.\\
$\mathrm{(1)}$ 
The family $\fb$ is admissible 
{\rm (}resp. universally admissible{\rm )}.\\
$\mathrm{(2)}$ 
If $\# S \geq 2$, 
then for any proper subset $W$ of $S$, 
any non-empty subset $U$ of $S$ 
such that $\# U\geq 2$, 
any elements $u\in U$ and $s\in S\ssm W$, 
a pair $(\fb_{U\ssm\{u\}}\wedge b_u\wedge \fa^{\wedge W},a_s)$ 
is distributive. 
{\rm (}Resp. 
If $\# S\geq 2$, 
then for any proper subset $W$ of $S$, 
any pair of disjoint non-empty subsets $U$ and $V$ of $S$ 
such that $\# V\geq 2$ 
and any elements $s\in S\ssm W$ and $v\in V$, 
a pair 
$(\fb_U\wedge \fb^{\wedge V}\wedge \fa^{\wedge W},
\fb^{\wedge V\ssm\{v\}}\wedge\fa^{\wedge W\coprod\{s\}})$ 
is distributive{\rm )}.\\
Then a family $\fb\wedge \fa^{\wedge S}$ 
is also admissible {\rm (}resp. universally admissible{\rm )}. 
\end{cor}

\begin{proof}[\bf Proof]
We set $r:=\# S$. 
We may assume without loss of generality that $S=(r]$ and $r\geq 2$. 
For any integers $k \in [-1,r-1]$ and $s\in S$, 
we set $c_s^{(k)}=b_s\wedge \fa^{\wedge [r-k,r]}$ 
and $\fc^{(k)}:=\{c_s^{(k)}\}_{s\in S}$. 
Notice that $\fc^{(-1)}=\fb$ and $\fc^{(r-1)}=\fb\wedge \fa^{\wedge S}$. 
\begin{claim}
For any integer $k\in [-1,r-1]$, 
the family $\fc^{(k)}$ is admissible 
{\rm (}resp. universally admissible{\rm )}.
\end{claim}
\sn
We prove the claim by induction on $k$. 
For $k=-1$, the assertion is 
nothing but assumption $\mathrm{(1)}$. 
Let us assume that 
the assertion is true for some integer $k \in [-1,r-2]$. 
Notice that we have the equality 
$c_s^{(k+1)}=c_s^{(k)}\wedge a_{r-k-1}$ for any $s\in (r]$. 
We apply Proposition~\ref{prop:adm seq lem} 
to the family $\fc^{(k)}$ and the element $a_{r-k-1}$. 
What we need to check is the following two conditions.\\
$\mathrm{(a)}$ 
The family $\fc^{(k)}$ 
is admissible (resp. universally admissible).\\
$\mathrm{(b)}$ 
For any element $s\in S$ and 
any non-empty subset $U$ of $S\ssm\{s\}$, 
a pair 
$(\fc_{U}^{(k)}\wedge c_s^{(k)},a_{r-k-1})$ 
is distributive. 
(Resp. For any pair of non-empty subsets 
$U$ and $V$ of $S$ and 
any element $v\in V$, 
a pair 
$(\fc_U^{(k)}\wedge {\fc^{(k)}}^{\wedge V},a_{r-k-1}
\wedge {\fc^{(k)}}^{\wedge V\ssm\{v\}})$ 
is distributive).\\
Condition $\mathrm{(a)}$ is just an inductive hypothesis 
and condition $\mathrm{(b)}$ 
follows from assumption $\mathrm{(2)}$. 
Hence the family $\fc^{(k+1)}=\fc^{(k)}\wedge a_{r-k-1}$ 
is admissible (resp. universally admissible), 
which completes the proof of the claim. 
Since $\fc^{(r-1)}=\fb\wedge \fa^{\wedge (r]}$, 
we obtain the desired result.  
\end{proof}

\begin{rem}
\label{rem:adm seq cor rem}
$\mathrm{(1)}$ 
In the situation 
Corollary~\ref{cor:adm seq cor} condition $\mathrm{(2)}$, 
we shall assume $s\neq u$ for the admissible case 
and $s\neq v$ for the universally admissible case.\\
$\mathrm{(2)}$ 
Moreover if we assume that 
$L$ is modular, 
then we shall assume that 
$W\cup U\neq S$ and that 
$s$ is not in $U$ for the universally admissible case. 
\end{rem}

\begin{proof}[\bf Proof] 
$\mathrm{(1)}$ 
For the admissible case (resp. the universally admissible case), 
if we assume $u=s$ (resp. $v=s$), 
then we have the (in)equality
$$b_k\wedge b_u\wedge \fa^{\wedge W}\leq a_s \ \ 
\text{(resp. $b_k\wedge \fb^{\wedge V}\wedge \fa^{\wedge W}\leq 
\fb^{\wedge V\ssm\{v\}}\wedge \fa^{\wedge W\coprod \{s\}} $)}.$$
Therefore we have the equality 
$$\left \{\underset{k\in U\ssm\{u\}}{\bigvee} 
(b_k\wedge b_u\wedge \fa^{\wedge W})  \right \}\wedge a_s= 
\underset{k\in U\ssm\{u\}}{\bigvee} 
(b_k\wedge b_u\wedge \fa^{\wedge W})$$
$$\text{(resp. 
$\left \{\underset{k\in U}{\bigvee}
(b_u\wedge \fb^{\wedge V}\wedge \fa^{\wedge W}) \right \}\wedge 
(\fa^{\wedge W\coprod\{s\}}\wedge \fb^{\wedge V\ssm\{v\}})= 
\underset{k\in U}{\bigvee}
(b_u\wedge \fb^{\wedge V}\wedge \fa^{\wedge W})
$)}.$$
Therefore a pair $(\fb_{U\ssm\{u\}}\wedge b_u\wedge \fa^{\wedge W},a_s)$ 
(resp. $(\fb_U\wedge \fb^{\wedge V}\wedge \fa^{\wedge W},
\fb^{\wedge V\ssm\{v\}}\wedge\fa^{\wedge W\coprod\{s\}})$) is distributive.

\sn
$\mathrm{(2)}$ 
Let us assume that $s$ is in $U$. 
Then since we have the inequality 
$$b_s\wedge \fb^{\wedge V}\wedge\fa^{\wedge W}\leq 
\fb^{\wedge V\ssm\{v\}}\wedge \fa^{\wedge W\coprod\{s\}},$$
we have the equalities
\begin{multline*}
\left \{\underset{k\in U}{\bigvee}
(b_u\wedge \fb^{\wedge V}\wedge \fa^{\wedge W}) \right \}\wedge 
(\fa^{\wedge W\coprod\{s\}}\wedge \fb^{\wedge V\ssm\{v\}})\\
=\left [ (b_k\wedge b_s\wedge \fa^{\wedge W})\vee 
\left \{\underset{k\in U\ssm\{s\}}{\bigvee}
(b_u\wedge \fb^{\wedge v}\wedge \fa^{\wedge W}) \right \}\right ]\wedge 
(\fa^{\wedge W\coprod\{s\}}\wedge \fb^{\wedge V\ssm\{v\}})\\
=(b_k\wedge b_s\wedge \fa^{\wedge W})\vee \left [ 
\left \{\underset{k\in U\ssm\{s\}}{\bigvee}
(b_u\wedge \fb^{\wedge v}\wedge \fa^{\wedge W}) \right \}\wedge 
(\fa^{\wedge W\coprod\{s\}}\wedge \fb^{\wedge V\ssm\{v\}}) \right ]
\end{multline*}
by the modularity of $L$. 
Hence we shall assume that $s$ is in $U$ by 
replacing $U\ssm\{s\}$ with $U$.
\end{proof}

\section{Cubes}
\label{sec:cubes} 

In this section, we introduce the notions of {\it cubes}. 
Let $S$ be a set, $P$ a partially ordered set and $\calD$ a category. 

\begin{df}[\bf Successor, Precessor]
\label{df:sucpre}
Let $x$ be an element in a partially ordered set $P$. 
A {\bf successor} (resp. {\bf predecessor}) 
of $x$ in $P$ is an element $t$ in $P$ such that 
$x<t$ (resp. $x>t$) and 
there exists no element $u$ in $P$ such that $x<u<t$ (resp. $x>u>t$). 
If $P$ is a totally ordered set, 
then the successor (resp. predecessor) of $x$ is uniquely determined 
if it exists and we denote it by $\Suc(x)$ (resp. $\Pre(x)$).
\end{df}

\begin{nt}
\label{nt:P^S}
The set of maps from $S$ to $P$ is denoted by $P^S$. 
We define the ordering $\leq$ on $P^S$ by $f\leq g$ 
if and only if $f(s)\leq g(s)$ for any element $s$ in $S$. 
Then $P^S$ is a partially ordered set. 
If $P$ is a lattice, then $P^S$ is also a lattice. 
Here for two elements $f$, $g$ in $P^S$, 
the maps $f\vee g$, $f\wedge g:S\to P$ send $s$ to 
$f(s)\vee g(s) $ and $f(s)\wedge g(s)$ respectively, 
for each element $s$ in $S$. 
(Comapare with {\bf Conventions} $\mathrm{(4)}$ $\mathrm{(iii)}$.) 
Notice that we have the equality as partially ordered sets 
\begin{equation}
\label{equ:dual P^S}
{(P^S)}^{\op}={(P^{\op})}^S.
\end{equation} 
\end{nt}

\begin{ex}[\bf (Double) Power sets]
\label{ex:power set}
$\mathrm{(1)}$
For any subset $T$ of $S$, 
we denote the {\bf characteristic function} 
({\bf of $T$ on $S$}) by $\chi_T:S \to [1]$. 
Namely $\chi_T(s)=1$ if $s$ is in $T$ and otherwise $\chi_T(s)=0$. 
We write $\cP(S)$ for the {\bf power set of $S$}. 
Namely $\cP(S)$ is the set of all subsets of $S$. 
We regard $\cP(S)$ as a partially ordered set 
ordered by set inclusion, 
a fortiori, a category. 
We also write $\cP'(S)$ for the set $\cP(S)\ssm\{\emptyset\}$. 
We have the canonical isomorphism of partially ordered sets 
\begin{equation}
\label{equ:powerset}
\cP(S)\isoto [1]^S
\end{equation}
which is defined by sending a subset $T$ of $S$ to the characteristic function $\chi_T$ of $T$ on $S$. 
If we regard $[1]$ as the Sierpinski space, 
namely the topolgical space whose class of open sets is 
$\{\emptyset,\ \{1\},\ \{0,\ 1\}\}$, 
and $S$ as a discrete topological space, 
then $\cP(S)\isoto [1]^S$ inherits the compact-open topolgy from $[1]$ and $S$. 
The class of open sets of $\cP(S)$ is the set 
of all ideals $\Ide(\cP(S))$.\\
$\mathrm{(2)}$ 
For any ordered pair of disjoint subsets $(U,V)$ of $S$, 
we define the {\bf characteristic function} ({\bf of  $(U,V)$ on $S$}) 
$\chi_{U,V}:S \to [2]$ as follows. 
For any element $s$ in $S$, 
$\chi_{U,V}(s)$ is $0$ if $s$ is in $S\ssm (U\coprod V)$, 
is $1$ if $s$ is in $U$ and is $2$ if $s$ is in $V$. 
We denote the set of all ordered pairs of disjoint subsets of $S$ by $\DP(S)$. 
Namely
$$\DP(S):=\{(U,V)\in\cP(S)\times\cP(S);U\cap V=\emptyset \}.$$
We define the ordering $\leq$ on $\DP(S)$ by declaring to be 
$(U,V)\leq (U',V')$ if and only if $V\subset V'$ and $U\subset U'\coprod V'$. 
Then $\DP(S)$ is a partially ordered set. 
We have the canonical isomorphsim of partially ordered sets 
\begin{equation}
\label{equ:doublepowerset}
\DP(S)\isoto [2]^S
\end{equation}
which is defined by sending an ordered pair of subsets $(U,V)$ of $S$ to 
the characteristic function $\chi_{U,V}$ of $(U,V)$ on $S$.\\
$\mathrm{(3)}$ 
For any pair of maps $(f,g)$ from $S$ to $[1]$, 
we define the map 
$f+g:S \to [2]$ by 
sending an element $s$ in $S$ to 
$f(s)+g(s)$. 
Then we have the map 
$${\textbf{+}}:[1]^S\times [1]^S \to [2]^S.$$
By the virtue of isomorphisms (\ref{equ:powerset}) and 
(\ref{equ:doublepowerset}), 
we also have the map
$${\textbf{+}}:\cP(S)\times \cP(S) \to \DP(S),\ \ 
(U,V)\mapsto U{\textbf{+}}V:=(U\ominus V,U\cap V).$$
For any element $T$ in $\cP(S)$, 
we write $e_T$ for the map from $\cP(S)$ to $\DP(S)$ 
which is sending a subset $U$ of $S$ to the element 
$U{\textbf{+}}T$ in $\DP(S)$. 
For any pair of disjoint subsets $(U,V)$ in $\DP(S)$ and 
for any disjoint decomposition $U=A\coprod B$, 
we have the equality
\begin{equation}
\label{equ:eUV 2}
(U,V)=e_{A\coprod V}(B\coprod V).
\end{equation}
\end{ex}

\begin{nt}
\label{nt:setin} 
Let $U$ be an element in $P^S$ and $s$ an element in $S$. 
We write $s\in \Supp U$ for the condition that 
$U(s)$ is not a minimal element in $P$. 
Now let us assume that $P$ is a totally ordered set, 
$s\in U$ and 
there exists the element $\Pre U(s)$. 
Then we define the map $U\ssm\{s\}:S \to P$ by putting 
that $U\ssm\{s\}(t)$ is $\Pre(U(s))$ if $s=t$ and, 
is $U(t)$ if $s\neq  t$. 
Then obviously we have the inequality $U\ssm\{s\} <U$. 
These notations are compatible with the usual ones when $P=[1]$. 
\end{nt}

\begin{df}[\bf Cubes]
\label{df:cubes}
An {\bf $(S,P)$-cube} (resp. {\bf $(S,P)$-cocube}) in a category $\calD$ 
is a contravariant (resp. covariant) functor 
from $P^S$ to $\calD$. 
We denote the category of $(S,P)$-cubes 
(resp. $(S,P)$-cocubes) in $\calD$ 
by $\Cub^{(S,P)}\calD$ (resp. $\CoCub^{(S,P)}\calD$). 
Here the morphisms between $(S,P)$-(co)cubes are just natural transformations. 
The associations $\calD \mapsto \Cub^{(S,P)}\calD$ and 
$\calD \mapsto \CoCub^{(S,P)}\calD$ 
give endofunctors on the category of (small) categories. 
For any $(S,P)$-(co)cube $x$ in $\calD$, any element $T$ in $P^S$, 
we write $x_T$ for $x(T)$ and call it the {\bf vertex of $x$} ({\bf at $T$}). 
We say that an $(S,P)$-cocube (resp. $(S,P)$-cube) $x$ in a category $\calD$ is {\bf monic} 
if for any pair of elements $U$ and $V$ in $P^S$ such that $U\leq V$ (resp. $V\leq U$), 
$x(U\leq V)$ is a monomorphism in $\calD$. 
Now assume that $P$ is a totally ordered set and 
let $s$ be an element in $S$ such that $s\in U$ and 
there exists the element $\Pre U(s)$ in $P$. 
We write $d_U^{s,x}$ or shortly $d_U^s$ for $x(U\ssm\{s\} <U)$ and call it 
the {\bf ($s$-)boundary morphism of $x$} ({\bf at $T$}).
\end{df}

\begin{ex}
\label{ex:various type cubes}
$\mathrm{(1)}$ 
If $P$ is a singleton $P=\{\ast\}$, 
then an $(S,P)$-(co)cube $x$ in a category $\calD$ is just a family 
$\{x_s\}_{s\in S}$ of objects in $\calD$ indexed by $S$.\\
$\mathrm{(2)}$ 
If $P=[1]$, 
then an $(S,P)$-cube (resp. $(S,P)$-cocube) $x$ 
in a category $\calD$ is regarded as 
a contravariant (resp. covariant) functor from $\cP(S)$ to $\calD$ 
by the isomorphism $\mathrm{(\ref{equ:powerset})}$. 
We simply call $(S,[1])$-(co)cubes {\bf $S$-(co)cubes} and 
we write $\Cub^S\calD$ (resp. $\CoCub^S\calD$) for 
$\Cub^{(S,[1])}\calD$ (resp. $\CoCub^{(S,[1])}\calD$).\\
$\mathrm{(3)}$ 
If $P=[2]$, 
then an $(S,P)$-cube (resp. $(S,P)$-cocube) 
$x$ in a category $\calD$ is regarded as 
a contravariant (resp. covariant) functor from $\DP(S)$ to $\calD$ 
by the isomorphism $\mathrm{(\ref{equ:doublepowerset})}$. 
We simply call $(S,[2])$-(co)cubes {\bf double $S$-(co)cubes} and 
we write $\DCub^S\calD$ (resp. $\CoDCub^S\calD$) for 
$\Cub^{(S,[2])}\calD$ (resp. $\CoCub^{(S,[2])}\calD$). 
For any double $S$-(co)cube $x$, 
any pair of disjoint subsets $(U,V)$ of 
$S$ and any element $s\in U\coprod V$, 
we write $x_{U,V}$ and $d_{U,V}^{s,x}$ (or shorty $d_{U,V}^s$) 
for $x_{(U,V)}$ and $d_{(U,V)}^{s,x}$. 
\end{ex}

\begin{df}[\bf Pull-back of cubes]
\label{df:pullback cubes}
Let $S$ and $T$ be sets, $P$ and $Q$ partially ordered sets 
and $f:P^S \to Q^T$ an order-preserving map. 
Then composition with $f$ induces 
the canonical natural transformations 
$f^{\ast}:\Cub^{(T,Q)}\to \Cub^{(S,P)}$ and 
$f^{\ast}:\CoCub^{(T,Q)}\to\CoCub^{(S,P)}$. 
For any $(T,Q)$-(co)cube $x$ in a category $\calD$, 
we call $f^{\ast}x$ the {\bf pull-back of $x$} 
({\bf along $f$}).
\end{df}

\begin{df}[\bf Attachment of objects to cubes]
\label{df:attachment of cube}
Let $S$ be a non-empty set, 
$P$ a partially ordered set with the minimum element $m$. 
We denote the minimum element in $P^S$ by $\emptyset$. 
Moreover let $x$ be an $(S,P)$-cube in a category $\calD$, 
$f:x_{\emptyset} \to y$ a morphism in $\calD$. 
We define an $(S,P)$-cube $x_{f,y}:{(P^S)}^{\op} \to \calD$ 
as follows. 
$x_{f,y}$ is equal to $x$ on $P^S\ssm\{\emptyset\}$ and 
we put $(x_{f,y})_{\emptyset}:=y$ and 
$x_{f,y}(\emptyset\leq U):=f\circ x(\emptyset\leq U)$ 
for any element $U$ in 
$P^S\ssm\{\emptyset\}$.
We call $x_{f,y}$ the {\bf attachment of $y$ to $x$} 
({\bf by $f$}).
\end{df}

\begin{nt}[\bf Dual (co)cube]
\label{nt:dual (co)cube}
Let $T$ be a finite totally ordered set. 
Then we have a unique isomorphism 
$D_T:T^{\op}\isoto T$ of partially ordered sets and 
the map $D_T$ induces an isomorphism of 
partially ordered sets
$${(T^S)}^{\op}={(T^{\op})}^S\overset{{(D_T)}^S}{\isoto} T^S .$$
Let $x$ be an $(S,T)$-cube 
(resp. $(S,T)$-cocube) in a category $\calD$. 
We define $\hat{x}$ the {\bf dual $(S,T)$-cocube} 
(resp. {\bf dual $(S,T)$-cube}) ({\bf of $x$}) by 
$\hat{x}:={(x{(D_T)}^S)}^{-1} $ (resp. $\hat{x}:=x{(D_T)}^S$).
\end{nt}

\begin{ex}[\bf Dual of $S$-(co)cubes]
\label{ex:dual of cube}
Let $x$ be an $S$-(co)cube in a category $\calD$. 
Then we have the equalities $\hat{x}_T=x_{S\ssm T}$ and 
$d_T^{t,\hat{x}}:=d_{(S\ssm T)\coprod\{t\}}^{t,x}$ 
for any $T\in\cP(S)$ and any $t\in T$. 
\end{ex}

\sn
The following lemma is sometimes useful 
for studying morphisms of cubes.

\begin{lem} 
\label{lem:genofcubemap} 
Let $S$ be a set and $P$ a totally ordered set 
and $x$ and $y$ $(S,P)$-cubes in a category $\calD$. 
Then\\
$\mathrm{(1)}$ 
Let $U$ and $V$ be elements in $P^S$ such that 
$U\leq V$ and the set $[U,V]$ is a finite set. 
Then the morphism $x(U \leq V)$ 
is described as compositions of boundary morphisms.\\
$\mathrm{(2)}$ 
Assume that $S$ and $P$ are  finite sets. 
Let 
$f={\{f_U:x_U \to y_U\}}_{U\in P^S}$ 
be a family of morphisms 
in $\calD$. 
Then $f:x \to y$ is a morphism of $(S,P)$-cubes in $\calD$ 
if and only if for any $U\in P^S$ and $s\in U$, 
we have the equality $d_U^{y,s}f_U=f_{U\ssm\{s\}}d_U^{x,s}$. 
\qed
\end{lem}

\begin{nt}
\label{nt:Pi_x}
Let $\calD$ be a category closed under finite limits, 
$\{x_s\onto{a_s} x\}_{s\in S}$ and 
$\{y_s\onto{b_s} y\}_{s\in S}$ families of morphisms 
to objects $x$ and $y$ in $\calD$ respectively indexed by 
a non-empty finite set $S=\{s_1,\cdots,s_r\}$, 
$f:x\to y$ a morphism in $\calD$ and 
$\{f_s:x_s\to y_s\}_{s\in S}$ 
a family of morphisms in $\calD$ indexed by $S$ 
such that $b_sf_s=fa_s$ for any element $s$ in $S$. 
Then we write $\underset{\!\!\!\!s\in S}{\prod_x} x_s$ and 
$\underset{\!\!\!\!s\in S}{\prod_{f}} f_s$ for 
$x_{s_1}\times_x x_{s_2}\times_x \cdots \times_x x_{s_r}$ and 
$f_{s_1}\times_f f_{s_2}\times \cdots \times_f f_{s_r}:x_{s_1}\times_x\cdots \times_x x_{s_r} \to y_{s_1}\times_y\cdots \times_y y_{s_r}$ respectively. 
\end{nt}

\begin{df}[\bf Coverings, associated cubes of coverings]
\label{df:P-th covering}
Let $P$ be a partially ordered set with the minimum element $m$, 
$S$ a non-empty finite set 
and $x$ an object in a category $\calD$.\\
$\mathrm{(1)}$ 
A {\bf $P$-covering} ({\bf of $x$ indexed by a non-empty set $S$}) 
is a family 
of contravariant functors 
$\fx:=\{x_s:P^{\op} \to \calD\}_{s\in S}$ 
such that $x_s(m)=x$ for any element $s$ in $S$.\\ 
$\mathrm{(2)}$ 
Let us assume that $S$ is a finite set and 
$\calD$ is closed under finite limits. 
Then for any $P$-covering 
$\fx=\{x^s:P^{\op}\to \calD\}_{s\in S}$ 
of an object $x$ in $\calD$ indexed by $S$, 
we can associate $\fx$ with 
the $(P,S)$-cube $\Fib\fx$ in $\calD$ as follows. 
For any elements $U$ and $V$ in $P^S$ such that $U\leq V$, 
we put ${(\Fib\fx)}_U:=\underset{\!\!\!\!s\in S}{\prod_x} x^s(U(s))$ and 
$(\Fib\fx)(U\leq V):=\underset{\!\!\!\!\!\!\!\!s\in S}{\prod_{\id_x}}x^s(U(s)\leq V(s))$. 
We call $\Fib\fx$ the 
{\bf $(S,P)$-cube associated with the $P$-covering $\fx$}.
\end{df}

\begin{df}[\bf Pull-back of coverings]
\label{df:pull-back covering}
Let $P$ and $Q$ be partially ordered sets 
with the minimum elements, 
$f:P \to Q$ an order preserving map preserving the minimum element, 
$x$ an object in a category $\calD$ and 
$\fx=\{x_s:Q^{\op} \to\calD\}_{s\in S}$ $Q$-covering of $x$ 
indexed by a non-empty set $S$. 
Then we put $f^{\ast}\fx:=\{x_sf^{\op}:P^{\op}\to\calD\}_{s\in S}$ 
and call it the {\bf pull-back of $\fx$} ({\bf along $f$}).
\end{df}

\begin{ex}
\label{ex:multiplication map}
For any positive integer $m$, 
the map $\mathbf{m}:[1] \to [m]$ which 
sends $0$ to $0$ and $1$ to $m$ is 
order preserving map preserving the minimum element. 
Therefore 
for any $[m]$-covering $\fx$ of an object $x$ in a category, 
we can define the $[1]$-covering $\mathbf{m}^{\ast}\fx$ of $x$. 
\end{ex}

\begin{nt}
\label{nt:abb of f^S}
Let $P$ and $Q$ be partially ordered sets and 
$f:P\to Q$ an order preserving map. 
For any set $S$, 
composition with $f$ induces the order preserving map 
$f^S:P^S \to Q^S$. 
If $P$ and $Q$ possess the minimum elements and 
$f$ preserves the minimum element, 
then $P^S$ and $Q^S$ also possess the minimum elements and 
the oredered map $f^S$ also preserves the minimum element. 
The map $f^S$ is sometimes abbreviated to $f$.
\end{nt}

\begin{df}[\bf Attachment of morphisms to coverings]
\label{df:attachment of coverings}
Let $P$ be a partially ordered set with the minimum element $m$, 
$\fx:=\{x^s:P^{\op}\to \calD\}_{s\in S}$ a $P$-covering 
of an object $x$ in a category $\calD$ indexed by 
a non-empty set $S$ and $f:x \to y$ a morphism in $\calD$.\\
$\mathrm{(1)}$ 
We define a partially ordered set 
$P^{\ast}:=P\coprod \{-\infty\}$ 
where $-\infty$ is a symbol and $p>-\infty$ for any 
element $p$ in $P$.\\
$\mathrm{(2)}$ 
We define $\iota_P:P\to P^{\ast}$ 
to be an order preserving map by sending an element 
$p$ to $p$ if $p\neq m$ and $m$ to $-\infty$.\\
$\mathrm{(3)}$ 
We define the $P^{\ast}$-covering 
$\fx_{f,y}:=\{x^s_{f,y}:{P^{\ast}}^{\op} \to \calD\}_{s\in S}$ 
of $y$ indexed by $S$ as follows. 
For any element $s$ in $S$, 
$x^s_{f,y}$ is equal to $x^s$ on $P$ and 
$x^s_{f,y}(-\infty)=y$ and 
$x^s_{f,y}(-\infty <p):=f\circ x^s(m\leq p)$ 
for any element $p$ in $P$.
\end{df}

\begin{ex}
\label{ex:compat atta and fib}
Let $P$ be a partially ordered set with the minimum element $m$, 
$\fx:=\{x^s:P^{\op}\to \calD\}_{s\in S}$ a $P$-covering 
of an object $x$ in a category $\calD$ indexed by 
a non-empty set $S$ and $f:x \to y$ a monomorphism in $\calD$. 
Let us assume that $\calD$ is closed under taking finite limits. 
Then we have the canonical isomorphism
\begin{equation}
\label{equ:att and fib}
\iota_P^{\ast}\Fib(\fx_{f,y})\isoto {(\Fib \fx)}_{f,y}.
\end{equation}
\end{ex}

\begin{ex}
\label{ex:definition of ux and cx}
Let $P$ be a partially ordered set with the minimum element $m$, 
$S$ a non-empty finite set 
and $x$ an $(S,P)$-cube in a category $\calD$.\\
$\mathrm{(1)}$ 
Recall the definition of the map $\emptyset:S \to P$ from 
Definition~\ref{df:attachment of cube}. 
It is the minimum element in $P^S$. 
Namely, it sends any element in $S$ to $m$.\\
$\mathrm{(2)}$ 
For any elements $p$ in $P$ and $s$ in $S$, 
we define a map $\delta_{s,p}:S \to P$ 
by sending an element $t$ in $S$ to $p$ if $s=t$ 
and to $m$ if $s\neq t$. 
Obviously for any elements $p\leq p'$, 
we have an inequality $\delta_{s,p}\leq \delta_{s,p'}$.\\
$\mathrm{(3)}$ 
We associate with $x$ a $P$-covering 
$\fU x:=\{x^s:P^{\op} \to \calD \}_{s\in S}$ of $x_{\emptyset}$ 
indexed by $S$ as follows. 
For any element $s$ in $S$, 
we define $x^s:P^{\op} \to \calD$ to 
be the functor which sends an element $p$ in $P$ to 
$x_{\delta_{s,p}}$ and any pair $p\leq p'$ 
in $P$ to $x(\delta_{s,p}\leq\delta_{s,p'})$.\\
$\mathrm{(4)}$ 
If $\calD$ is closed under finite limits and $S$ is a non-empty finite set, 
then we have the canonical morphism of $(S,P)$-cubes 
$C(x):x \to \Fib \fU x$ which is 
induced from the identity morphisms on $x_{\delta_{s,p}}$ 
for any elements $s$ in $S$ and $p$ in $P$ 
by the universal property of fiber products.
\end{ex}

\begin{df}[\bf Fibered cubes]
\label{df:fibered S,P-cubes}
Let $P$ be a partially ordered set with 
the minimum element $m$, 
$S$ a non-empty finite set, 
$\calD$ a category closed under finite limits. 
An $(S,P)$-cube $x$ in $\calD$ is {\bf fibered} if 
the canonical morphism of 
$(S,P)$-cubes $C(x):x \to \Fib \fU x$ is an isomorphism.
\end{df}

\begin{lem}[\bf Compatibility of pull-backs]
\label{lem:comp pull-back}
Let $P$ and $Q$ be partially ordered sets 
with the minimum elements, 
$f:P\to Q$ 
an order-preserving map which preserves the minimum element, 
$S$ a non-empty set, 
$\calD$ a category closed under finite limits.\\
$\mathrm{(1)}$ 
For a $Q$-covering indexed by $S$, 
$\fx:=\{x^s:Q^{\op}\to\calD\}_{s\in S}$ 
of an object $z$ in $\calD$, 
we have the canonical isomorphism of $(P,S)$-cubes
$$\Fib(f^{\ast}\fx)\isoto f^{\ast}\Fib \fx. $$
$\mathrm{(2)}$ 
For a $(Q,S)$-cube $x$ in $\calD$, 
we have the canonical equality of $P$-coverings
$$f^{\ast}\fU x=\fU f^{\ast}x.$$
$\mathrm{(3)}$ 
For a fibered $(Q,S)$-cube $y$ in $\calD$, 
the pull-back $f^{\ast}y$ of $y$ along $f$ 
is a fibered $(P,S)$-cube.
\end{lem}

\begin{proof}[\bf Proof]
For any elements $s$ in $S$ and $p$ in $P$ and 
any map $U$ from $S$ to $P$, 
we have the canonical equalities 
$${\Fib (f^{\ast}x)}_U=\underset{\!\!\!\!s\in S}{{\prod}_z}x^s(fU(s))=
{(f^{\ast}\Fib x)}_U \ \ \ \text{and}$$
$$f\delta_{s,p}=\delta_{s,f(p)}.$$
Assertions $\mathrm{(1)}$ and $\mathrm{(2)}$ follow from the 
equalities above respectively. 
For assertion $\mathrm{(3)}$, 
we have the commutative diagram of $(P,S)$-cubes. 
$$
{\footnotesize{
\xymatrix{
f^{\ast}y \ar[r]^{\!\!\!\!C(f^{\ast}y)} 
\ar[d]^{\wr}_{f^{\ast}C(y)} & 
\Fib\fU f^{\ast}y \ar@{=}[d]^{\textbf{I}}\\
f^{\ast}\Fib y \ar[r]^{\!\!\!\!\sim}_{\!\!\!\!\textbf{II}} & 
\Fib f^{\ast}\fU y.
}}}
$$
Since the morphisms 
$f^{\ast}C(y)$, $\textbf{I}$ and $\textbf{II}$ 
are isomorphisms by assumption and assertions 
$\mathrm{(1)}$ and $\mathrm{(2)}$, 
$C(f^{\ast}y)$ is also isomorphism. 
Therefore $f^{\ast}y$ is also fibered. 
\end{proof}

\begin{df}[\bf Power sets of cubes and coverings] 
\label{df:powerset of cubes}
Assume that $P$ has the minimum element $m$.\\
$\mathrm{(1)}$ 
Let $x$ be a monic $P^S$-cube in an abelian category $\cA$. 
Then we may ragard all vertices of $x$ as subobjects of $x_{\emptyset}$. 
We write $\cP(x)$ for the sublattice of $\cP(x_{\emptyset})$ 
generated by all vertices of $x$.\\
$\mathrm{(2)}$ 
Let $z$ be an object in $\cA$ and $\fz$ is a $P$-covering of $z$. 
Assume that $\Fib \fz$ is monic. 
Then we write $\cP(\fz)$ for $\cP(\Fib \fz)$. 
\end{df}

\begin{lem}[\bf Caracterization of fibered cubes] 
\label{lem:char of fibered S-cube}
Let $S$ be a finite set such that $\# S\geq 2$, 
then the following conditions are equivalent 
for any $S$-cube $x$ in a category $\calD$ closed under finite limits.\\
$\mathrm{(1)}$ 
The $S$-cube $x$ is fibered.\\
$\mathrm{(2)}$ 
For any pair $(s,t)$ of distinct elements in $S$ and 
for any subset $T\subset S\ssm\{s,t\}$, 
the commutative diagram
$${\footnotesize{\xymatrix{
x_{T\coprod\{s,t\}} 
\ar[r]^{d_{T\coprod\{s,t\}}^{t,x}} 
\ar[d]_{d_{T\coprod\{s,t\}}^{s,x}} & 
x_{T\coprod\{s\}} \ar[d]^{d_{T\coprod\{s\}}^{s,x}}\\
x_{T\coprod\{t\}} \ar[r]_{d_{T\coprod\{t\}}^{t,x}} & 
x_T
}}}$$
is a Cartesian square.
\end{lem}

\begin{proof}[\bf Proof] 
Obviously assertion $\mathrm{(1)}$ implies 
assertion $\mathrm{(2)}$. 
We prove the converse implication. 
Namely we prove that $C(x)_T$ is an isomorphism for any $T$ by 
induction on the cardinality of $T$. 
If $\# T\leq 1$, then $C(x)_T$ is the identity morphism. 
For any subset $T$ of $S$ such that $\# T\geq 2$, 
we fix a pair of distinct elements $s$ and $t$ in $T$. 
By hypothesis, 
the square 
$${\footnotesize{\xymatrix{ 
x_T \ar[r]^{d^{s,x}_T} \ar[d]_{d^{t,x}_T} & 
x_{T\ssm\{s\}} \ar[d]^{d^{t,x}_{T\ssm\{s\}}}\\
x_{T\ssm\{t\}} \ar[r]_{d^{t,x}_{T\ssm\{t\}}} & 
x_{T\ssm\{s,t\}}
}}}$$
is a Cartesian square. 
Since $C(x)_{U}$ is an isomorphism 
for any proper subset $U$ of $T$ by the inductive hypothesis,
$C(x)_T=C(x)_{T\ssm\{t\}}\times_{C(x)_{T\ssm\{s,t\}}}C(x)_{T\ssm\{s\}}$ 
is also an isomorphism 
by the universal property of the fiber product $(\Fib \fU x)_T$. 
Hence we get the desired assertion. 
\end{proof}

\begin{ex}
\label{ex:ass double cubes subobj}
Let $\calD$ be a category closed under finite limits, 
$S$ a non-empty finite set, 
$x$ an object in $\calD$, 
$\fx=\{x^s:[2]^{\op} \to \calD\}_{s\in S}$ 
a $[2]$-covering of $x$ indexed by $S$ 
and $T$ a subset of $S$. 
We set 
$\fx|_T:=\fU e_T^{\ast}\Fib\fx$, 
which is a $[1]$-covering of 
${(\Fib \fx)}_{T,\emptyset}$ 
indexed by $S$, 
where $e_T$ is as in Example~\ref{ex:power set} $\mathrm{(3)}$. 
Then we have the canonical isomorphisms 
\begin{equation}
\label{equ:compt res and fib}
e_T^{\ast}(\Fib \fx)\isoto \Fib (\fx|_T)
\end{equation}
\begin{equation}
\label{equ:comt res and fib }
\mathbf{2}^{\ast}(\Fib \fx) \isoto \Fib ({\mathbf{2}}^{\ast}\fx)
\end{equation}
by Lemma~\ref{lem:comp pull-back} $\mathrm{(1)}$ 
and $\mathrm{(3)}$. 
Here $\mathbf{2}$ is as in Example~\ref{ex:multiplication map}.
\end{ex}

\begin{lem}[\bf Characterization of fibered double cubes]
\label{lemdf:fib dou cubes}
Let $x$ be a double $S$-cube in a category 
$\calD$ closed under finite limits.\\
$\mathrm{(1)}$ 
For any subset $T$ of $S$, we have the commutative diagrams 
$${\footnotesize{
\xymatrix{
e_T^{\ast}x \ar[r]^{e_T^{\ast}C(x)} \ar[d]_{C(e_T^{\ast}x)} & 
e_T^{\ast}(\Fib \fU x) \ar[d]^{\wr}\\
\Fib (\fU (e_T^{\ast}x)) \ar[r]_{\sim} & 
\Fib((\fU x)|_T)
}
\xymatrix{
\mathbf{2}^{\ast}x \ar[r]^{\mathbf{2}^{\ast}C(x)} 
\ar[d]_{C(\mathbf{2}^{\ast}x)} & 
\mathbf{2}^{\ast}(\Fib \fU x) \ar[d]^{\wr}\\
\Fib (\fU (\mathbf{2}^{\ast}x)) \ar[r]_{\sim} & 
\Fib ({\mathbf{2}^{\ast}}(\fU x)). 
}}}$$
$\mathrm{(2)}$ 
If $\# S\geq 2$, 
then the following conditions are equivalent:\\
$\mathrm{(i)}$ 
The $S$-cubes 
$e_T^{\ast}x$ is fibered $S$-cubes for any subset $T$ of $S$.\\
$\mathrm{(ii)}$ 
The $S$-cubes 
$\mathbf{2}^{\ast}x$ and $e_T^{\ast}x$ are fibered $S$-cubes for any proper subset $T$ of $S$.\\
$\mathrm{(iii)}$ 
The double $S$-cube $x$ is fibered.\\
$\mathrm{(iv)}$ 
The canonical morphism $e_T^{\ast}C(x)$ 
is an isomorphism of $S$-cubes for any subset $T$ of $S$.\\
$\mathrm{(v)}$ 
The canonical morphisms $\mathbf{2}^{\ast}C(x)$ and $e_T^{\ast}C(x)$ are 
isomorphisms of $S$-cubes for any proper subset of $S$.
\end{lem}

\begin{proof}[\bf Proof]
Assertion $\mathrm{(1)}$ is straightforward. 
In assertion $\mathrm{(2)}$, 
conditions $\mathrm{(iii)}$, $\mathrm{(iv)}$ 
and $\mathrm{(v)}$ are obviously equivalent. 
By the virtue of the assertion $\mathrm{(1)}$ 
and Lemma~\ref{lem:char of fibered S-cube}, 
condition $\mathrm{(i)}$ (resp. $\mathrm{(ii)}$) 
is equivalent to condition $\mathrm{(iv)}$ (resp. $\mathrm{(v)}$). 
Hence we get assertion $\mathrm{(2)}$. 
\end{proof}

\begin{cor}
\label{cor:fib discribe}
Let $x$ be a monic double $S$-cube in 
an abelian category $\cA$. 
We put $\fa:=\{x_{\{s\},\emptyset} \}_{s\in S}$ 
and $\fb:=\{x_{\emptyset,\{s\}} \}_{s\in S}$. 
Moreover 
let us assume that $\mathbf{2}^{\ast}x$ 
and $e_T^{\ast}x$ are fibered for any proper subset $T$ of $S$. 
Then we have the canonical isomorphism 
$x_{U,V}\isoto \fa^{\wedge U}\wedge\fb^{\wedge V}$ 
as subobjects of $x_{\emptyset,\emptyset}$ 
for any $(U,V)$ in $\DP(S)$. 
Here the symbol $\wedge$ means the meet 
in the lattice $\cP(x)$. 
\end{cor}

\begin{proof}[\bf Proof] 
We shall only remark that in this case, 
fiber products of vertices of $x$ over $x_{\emptyset,\emptyset}$ is just 
the wedge products in $\cP(x)$. 
The assertions follow from Lemma~\ref{lemdf:fib dou cubes}.
\end{proof}

\section{Admissible cubes}
\label{sec:adm cubes}

In this section, we review some notations of 
admissible cubes introduced in \cite{Kos}. 
Let $S$ be a finite set and $\cA$ an abelian category.

\begin{df}[\bf Restriction of cubes] 
\label{para:rest of cubes}
Let $U$, $V$ and $W$ be a triple of subsets of $S$ 
such that $U\cap V=\emptyset$ and $U\coprod V\subset W$. 
We define $i_{U,W}^V:\cP(U) \to \cP(W)$ 
to be the functor which sends 
a set 
$A \in \cP(U)$ 
to the disjoint union set $A\coprod V$ of $A$ and $V$. 
Composition with $i^{V}_{U,S}$ induces the natural transformation 
${(i_{U,S}^V)}^{\ast}:\Cub^S \to \Cub^U$. 
(See Definition~\ref{df:pullback cubes}.) 
For any $S$-cube $x$ in a category $\calD$, 
we write $x|_U^V$ for ${(i_{U,S}^V)}^{\ast}x$ 
and call it 
the {\bf restriction of $x$} ({\bf to $U$ along $V$}). 
For any pair of disjoint subsets $U$ and $V$ of $S$ 
and subsets $A\subset U$ and $B\subset V$, 
we have the following commutative diagram.
$${\footnotesize{\xymatrix{
\cP(S\ssm(U\coprod V)) \ar[r]^{i^A_{S\ssm(U\coprod V),S\ssm U}} 
\ar[rd]^{i^{A\coprod B}_{S\ssm(U\coprod V),S}}
\ar[d]_{i^B_{S\ssm(U\coprod V),S\ssm V}} & 
\cP(S\ssm U) \ar[d]^{i^B_{S\ssm U,S}} \\
\cP(S\ssm V) \ar[r]_{i_{S\ssm V,S}^A} & 
\cP(S).
}}}$$
In particular, for any $S$-cube $x$ in $\cC$, 
we have the equalities 
\begin{equation}
\label{equ:restr cube}
(x|_{S\ssm V}^A)|^B_{S\ssm (U\coprod V)}=
x|^{A\coprod B}_{S\ssm(U\coprod V)}=
(x|^B_{S\ssm U})|^A_{S\ssm(U\coprod V)}.
\end{equation}
\end{df}

\begin{ex}[\bf Faces of cubes]
\label{ex:Faces of cubes}
For any $S$-cube $x$ in a category $\calD$ 
and any $k\in S$, 
$x|_{S\ssm\{k\}}^{\{k\}}$ and $x|_{S\ssm\{k\}}^{\emptyset}$ are called 
the {\bf backside $k$-face of $x$} and 
the {\bf frontside $k$-face of $x$} respectively. 
By a {\bf face} of $x$, 
we mean any backside or frontside $k$-face of $x$. 
We write $d^{k,x}:x|_{S\ssm\{k\}}^{\{k\}} \to x|^{\emptyset}_{S\ssm\{k\}}$ for 
the natural transformation from $x|^{\{k\}}_{S\ssm\{k\}}$ to $x|^{\emptyset}_{S\ssm\{k\}}$ 
induced by the boundary morphisms $d^{k,x}_{T\coprod\{k\}}:x_{T\coprod\{k\}} \to x_T$ 
for any $T\in \cP(S\ssm\{k\})$. 
\end{ex}

\begin{ex}
\label{ex:ass cubes subobj} 
Let $\calD$ be a category closed under finite limits and 
$x$ an object in $\calD$ and 
$\fx=\{x_s \onto{d_s^x} x\}_{s\in S}$ a family of morphisms to $x$. 
For any element $s\in S$, we obviously have the canonical isomorphism 
\begin{equation}
\label{equ:face of fib 1}
\Fib (\fx)|_{S\ssm\{s\}}^{\emptyset}\isoto\Fib (\fx_{S\ssm\{s\}}).
\end{equation}
Moreover if all $d_s^x$ are monomorphisms, 
then we also have the canonical isomorphism
\begin{equation}
\label{equ:face of fib 2}
\Fib (\fx)|_{S\ssm\{s\}}^{\{s\}}\isoto\Fib (\fx_{S\ssm\{s\}}\wedge x_s)
\end{equation}
where the symbol $\wedge$ means the meet in $\cP(x)$. 
\end{ex}

\begin{df}[\bf Composition of cubes] 
\label{df:comp of cubes}
Let $x$ and $y$ be $S$-cubes in a category $\calD$, 
$s$ an element in $S$ and 
$\alpha:x|_{S\ssm\{s\}}^{\emptyset} \to y|_{S\ssm\{s\}}^{\{s\}}$ 
a morphism of $S\ssm\{s\}$-cubes. 
We define $x\circ_{s,\alpha} y$ to be an $S$-cube as follows. 
Let $T$ be a subset of $S$ and $t$ an element in $T\ssm\{s\}$. 
${(x\circ_{s,\alpha} y)}_T$ is $x_T$ if $s$ is in $T$, 
and is $y_T$ if $s$ is not in $T$. 
The morphism $d_T^{t,x\circ_{s,\alpha} y}:
{(x\circ_{s,\alpha}y)}_T \to {(x\circ_{s,\alpha}y)}_{T\ssm\{t\}}$ 
is $d_T^{t,x}$ if $s$ is in $T\ssm\{t\}$, 
is $d_T^{t,y}$ if $s$ is not in $T$, 
and we put 
$d_T^{s,x\circ_{s,\alpha}y}:=d_T^{s,x}\alpha d_T^{s,y}$. 
We call $x\circ_{s,\alpha}y$ 
the {\bf composition of $S$-cubes $x$ and $y$} 
({\bf along $s$-direction by $\alpha$}).
\end{df}

\begin{nt}[\bf Total complexes]
\label{nt:tot complex}
Let $S$ be a non-empty finite set such that $\# S=n$ and 
$x$ an $S$-cube in an additive category $\cB$. 
Let us fix a bijection $\alpha$ from $S$ to $(n]$ 
and we will identify $S$ with the set $(n]$ via $\alpha$. 
We associate an $S$-cube $x$ with 
{\bf total complex} $\Tot_{\alpha} x=\Tot x$ 
as follows. 
$\Tot x$ is a chain complex in $\cB$ concentrated in 
degrees $0,\ldots,n$ whose component at degree $k$ is given by
$$\displaystyle{{(\Tot x)}_k:=\underset{\substack{T\in\cP(S) \\ 
\# T=k}}{\bigoplus} x_T}$$ 
and whose boundary morphism 
$d_k^{\Tot x}:{(\Tot x)}_k \to {(\Tot x)}_{k-1}$ 
are defined by 
$$(-1)^{\overset{n}{\underset{t=j+1}{\sum}}\chi_T(t)}d_T^j:x_T 
\to x_{T\ssm\{j\}}$$
on its $x_T$ component to $x_{T\ssm\{j\}}$ component. 
Here $\chi_T$ is the characteristic function of $T$. 
(See {\bf Example}~\ref{ex:power set}.) 
\end{nt}

\begin{ex}[\bf Mapping cone]
\label{ex:mapping cone}
Let $\cB$ be an additive category. 
For a chain morphism between chain complexes $f:a\to b$ in $\cB$, 
we denote the {\bf canonical mapping cone of $f$} by $\Cone f$. 
Namely $\Cone f$ is a complex in $\cB$ and whose component at degree $n$ is given by 
${(\Cone f)}_n=a_{n-1}\oplus b_n $ 
and whose boundary morphism 
$d_n^{\Cone f}:{(\Cone f)}_n \to {(\Cone f)}_{n-1}$ 
are defined by 
$\displaystyle{d_n^{\Cone f}=
\begin{pmatrix}
-d_{n-1}^a & 
0\\ 
-f_{n-1} & 
d_n^b 
\end{pmatrix}}$. 
Let $x$ be an $S$-cube in $\cB$. 
Then for any $s\in S$, we have the canonical isomorphism 
\begin{equation}
\label{eq:cone isom}
\Cone(\Tot d^{s,x}:\Tot(x|^{\{s\}}_{S\ssm\{s\}}) \to 
\Tot(x|^{\emptyset}_{S\ssm\{s\}}))\isoto \Tot x
\end{equation}
In particular, 
if $\cB$ is an abelian category, 
then we have the long exact sequence
\begin{multline}
\label{eq:long ex seq}
\cdots \to \Homo_{p+1}\Tot x \to \Homo_p \Tot (x|^{\{s\}}_{S\ssm\{s\}}) \to 
\Homo_p\Tot(x|^{\emptyset}_{S\ssm\{s\}}) \to\\ 
\Homo_p\Tot x \to 
\Homo_{p-1}\Tot (x|^{\{s\}}_{S\ssm\{s\}})\to 
\Homo_{p-1}\Tot (x|^{\emptyset}_{S\ssm\{s\}}) \to \cdots .
\end{multline}
\end{ex}

\begin{df}[\bf Spherical complexes, spherical cubes]
\label{df:spherical complex}
Let $n$ be an integer. 
We say that a complex $y$ 
in an abelian category $\cA$ is {\bf $n$-spherical} if 
$\Homo_k(y)=0$ for any $k\neq n$. 
We say that an $S$-cube $x$ in an abelian category $\cA$ 
is {\bf $n$-spherical} if the complex $\Tot x$ is $n$-spherical. 
\end{df}

\sn 
By the long exact sequence $\mathrm{(\ref{eq:long ex seq})}$ 
in Example~\ref{ex:mapping cone}, 
we can easily get the following result.

\begin{lem}[\bf Homology groups of Total complexes]
\label{lem:Cal of Tot}
Let $x$ be an $S$-cube in an abelian category $\cA$ and 
let us assume that $x|_{S\ssm\{s\}}^{\emptyset}$ is $0$-spherical, 
then we have the canonical isomorphisms 
\begin{equation}
\Homo_p\Tot x\isoto 
\begin{cases}
\coker\Homo_0\Tot d^{s,x} & \text{{\rm{if $p=0$}}}\\
\Ker\Homo_0\Tot d^{s,x} & \text{{\rm{if $p=1$}}}\\
\Homo_{p-1}\Tot(x|^{\{s\}}_{S\ssm\{s\}}) & \text{{\rm{if $p\geq 2$}}}.
\end{cases}
\end{equation} 
Here $d^{s,x}$ is as in Example~\ref{ex:Faces of cubes}. 
\qed
\end{lem}

\begin{ex}
\label{ex:H_0Totds}
Let $x$ be an $S$-cube in $\cA$ such that 
all boundary morphisms are monomorphisms. 
Then for any $s\in S$, we have the isomorphism
\begin{equation}
\label{equ:H0Totds}
\Ker(\Homo_0\Tot d^{s,x}:\Homo_0\Tot x|_{S\ssm\{s\}}^{\{s\}} 
\to \Homo_0\Tot x|_{S\ssm\{s\}}^{\emptyset})
\isoto 
\frac{\left (\underset{t\in S\ssm\{s\}}{\bigvee} 
\im d_{\{t\}}^{t,x}\right )\wedge\im d_{\{s\}}^{s,x}}
{\underset{t\in S\ssm\{s\}}{\bigvee} 
(\im d_{\{t\}}^{t,x}\wedge \im d_{\{s\}}^{t,x})}
\end{equation}
where the symbols $\vee$ and $\wedge$ are 
the join and the meet in $\cP(x_{\emptyset})$ respectively. 
Therefore the morphism $\Homo_0\Tot d^{s,x}$ 
is a monomorphism if and only if 
a pair of (a family of) subobjects 
$(\{\im d_{\{t\}}^{t,x}\}_{t\in S\ssm\{s\}},\im d^{s,x}_{\{s\}})$ 
is distributive in $\cP(x_{\emptyset})$. 
\end{ex}

\begin{ex}
\label{ex:sphericality of comp}
Let $x$ and $y$ be $S$-cubes in an additive category 
$\cB$, 
$s$ an element in $S$ 
and 
$\alpha:x|_{S\ssm\{s\}}^{\emptyset} 
\isoto y|^{\{s\}}_{S\ssm\{s\}}$ 
an isomorphism of $S$-cubes. 
Then we have a sequence of complexes of $S\ssm\{s\}$-cubes.
$$
\begin{bmatrix}
\xymatrix{
x|^{\{s\}}_{S\ssm\{s\}} \ar[d]^{d^{s,x}}\\
x|^{\emptyset}_{S\ssm\{s\}}
}
\end{bmatrix}
\begin{matrix}
\xymatrix{
\overset{\scriptstyle{\id}}{\to} \ar@{}[d]\\
\underset{\scriptstyle{d^{s,y}\alpha}}{\to}\\
}
\end{matrix}
\begin{bmatrix}
\xymatrix{
x|^{\{s\}}_{S\ssm\{s\}} \ar[d]^{d^{s,y}\alpha d^{s,x}} \\
y|^{\emptyset}_{S\ssm\{s\}}
}
\end{bmatrix}
\begin{matrix}
\xymatrix{
\overset{\scriptstyle{\alpha d^{s,x}}}{\to} \ar@{}[d]\\
\underset{\scriptstyle{\id}}{\to}
}
\end{matrix}
\begin{bmatrix}
\xymatrix{
y|^{\{s\}}_{S\ssm\{s\}} \ar[d]^{d^{s,y}}\\
y|^{\emptyset}_{S\ssm\{s\}}
}
\end{bmatrix}.
$$
We regard the sequence above as 
\begin{equation}
\label{equ:dist tri}
x \to x\circ_{s,\alpha} y \to y
\end{equation}
and it is a distinguished triangle 
in the triangulated category 
of the homotopy category of 
chain complexes of $S\ssm\{s\}$-cubes 
by the octahedron axiom. 
In particular if $\cB$ is an abelian category, we have 
a long exact sequence
\begin{multline}
\label{eq:long ex seq 2}
\cdots \to \Homo_{p+1}\Tot(y) \to \Homo_p \Tot (x) \to 
\Homo_p\Tot(x\circ_{s,\alpha} y) \to\\ 
\Homo_p\Tot(y) \to 
\Homo_{p-1}\Tot (x)\to 
\Homo_{p-1}\Tot (x\circ_{s,\alpha} y) \to \cdots .
\end{multline} 
In particular, 
let $n$ be a non-negative integer and 
if $x$ and $y$ are $n$-spherical, 
then $x\circ_{s,\alpha} y$ is also $n$-spherical.  
\end{ex}

\begin{df}[\bf Homology of cubes]
\label{df:homology of cubes}
Let us fix an $S$-cube $x$ 
in $\cA$. 
For each $k \in S$, 
the {\bf $k$-direction $0$-th homology} of $x$ 
is the $S\ssm\{k\}$-cube $\Homo_0^k(x)$ in $\cA$ 
defined by $\Homo_0^k(x)_T:=\coker d_{T\cup\{k\}}^k$. 
For any $T\in \cP(S)$ and $k\in S\ssm T$, 
we denote 
the canonical projection morphism 
$x_T \rdef \Homo_0^k(x)_T$ 
by $\pi^{k,x}_T$ or simply $\pi^k_T$. 
\end{df}

\begin{ex}[\bf Motivational example]
\label{ex:mot ex}
Let $\ff_S=\{f_s\}_{s\in S}$ be a family of elements in $A$. 
The {\bf typical cube} associated 
with $\ff_S=\{f_s\}_{s\in S}$ 
is an $S$-cube in 
the category of $A$-modules 
denoted by
$\Typ_A(\ff_S)$ 
and defined by 
$\Typ_A(\ff_S)_T=A$ and $d_T^{\Typ_A(\ff_S),t}=f_t$ for any 
$T\in \cP(S)$ and $t\in T$. 
The complex $\Tot \Typ_A(\ff_S)$ is the usual Koszul complex associated with 
a family $\ff_S$. 
If the family $\ff_S$ forms a regular sequence 
with respect to every ordering of the 
members of $\ff_S$, 
then for any $k \in (\# S]$ and any 
distinct elements $s_1,\cdots,s_k$ in $S$, 
boundary maps of 
$\Homo_0^{s_1}(\cdots(\Homo_0^{s_k}(\Typ_A(\ff_S)))\cdots)$ 
are injections. 
\end{ex}

\begin{df}[\bf Admissible cubes]
\label{para:admcube}
Let us fix an $S$-cube $x$ 
in $\cA$. 
When $\# S=1$, 
we say that $x$ is {\bf admissible} if $x$ is monic, 
namely if its unique boundary morphism is a monomorphism. 
For $\# S>1$, 
we define the notion of an admissible cube inductively 
by saying that 
$x$ is {\bf admissible} if $x$ is monic 
and if for every $k$ in $S$, 
$\Homo^k_0(x)$ is admissible. 
If $x$ is admissible, 
then for any distinct elements 
$i_1,\ldots,i_k$ in $S$ 
and for any automorphism $\sigma$ of $S$, 
the identity morphism on $x|^{\emptyset}_{S\ssm\{i_1,\ldots,i_k\}}$ 
induces an isomorphism
\begin{equation}
\label{equ:ordering}
\Homo^{i_1}_0(\Homo^{i_2}_0(\cdots(\Homo_0^{i_k}(x))\cdots))\isoto 
\Homo^{i_{\sigma(1)}}_0(\Homo^{i_{\sigma(2)}}_0(\cdots(
\Homo_0^{i_{\sigma(k)}}(x))\cdots))
\end{equation}
(\cf \cite[3.11]{Kos}). 
For an admissible $S$-cube $x$ and 
a subset $T=\{i_1,\ldots,i_k\}\subset S$, 
we set 
$\Homo^T_0(x):=\Homo^{i_1}_0(\Homo^{i_2}_0(\cdots(\Homo^{i_k}_0(x))\cdots))$ 
and $\Homo^{\emptyset}_0(x)=x$. 
By virtue of the isomorphism $\mathrm{(\ref{equ:ordering})}$, 
the definition of $\Homo_0^T(x)$ does not depend upon 
an ordering of the sequence $i_1,\ldots,i_k$, 
up to isomorphisms. 
Notice that $\Homo^T_0(x)$ is an $S\ssm T$-cube for any $T\in\cP(S)$. 
We have the isomorphisms
\begin{equation}
\label{equ:Totisom}
\Homo_p(\Tot(x))\isoto 
\begin{cases}
\Homo^S_0(x) & \text{for $p=0$}\\
0 & \text{otherwise}.
\end{cases}
\end{equation}
(See \cite[3.13]{Kos}.) 
We denote the full subcategory of $\Cub^S(\cA)$ consisting of 
those admissible cubes by 
$\Cub_{\adm}^S(\cA)$. 
The category $\Cub_{\adm}^S(\cA)$ is 
closed under extensions in $\Cub^S(\cA)$ 
by \cite[3.20]{Kos} and therefore it naturally becomes an exact category. 
\end{df}

\begin{ex}[\bf Admissible squares]
\label{ex:adm squ}
Let $x$ be a $\{1,2\}$-cube in an abelian category such that 
the boundary morphisms $d_{\{1\}}^{1,x}$ and 
$d_{\{2\}}^{2,x}$ are monomorphisms. 
Then we can easily prove that the following conditions are equivalent.\\
$\mathrm{(1)}$ 
The cube $x$ is $0$-spherical.\\
$\mathrm{(2)}$ 
The cube $x$ is admissible.\\
$\mathrm{(3)}$ 
The diagram
$$
{\footnotesize{\xymatrix{
x_{\{1,2\}} \ar[r]^{d^{2,x}_{\{1,2\}}} \ar[d]_{d^{1,x}_{\{1,2\}}} & 
x_{\{1\}} \ar[d]^{d_{\{1\}}^{1,x}}\\
x_{\{2\}} \ar[r]_{d^{2,x}_{\{2\}}} & 
x_{\emptyset}
}}}$$
is a Cartesian square.
\end{ex}

\begin{thm}[\bf Characterization of admissibility]
\label{thm:char of adm}
{\rm (cf. \cite[3.15]{Kos}).} 
Let $x$ be an $S$-cube in an abelian category $\cA$. 
Then the following conditions are equivalent.\\
$\mathrm{(1)}$ 
The $S$-cube $x$ is admissible.\\
$\mathrm{(2)}$ 
All faces of the $S$-cube $x$ are admissible 
and the $S$-cube $x$ is $0$-spherical.\\
$\mathrm{(3)}$ 
All frontside faces of the $S$-cube $x$ are admissible 
and the $S$-cube $x$ is $0$-spherical.\\
$\mathrm{(4)}$ 
The cube $x|_T^{\emptyset}$ is admissible for any $T\in \cP(S)$.
\end{thm}

\begin{proof}[\bf Proof]
The equivalence of conditions $\mathrm{(1)}$ and 
$\mathrm{(2)}$ is proven in \cite[3.15]{Kos}. 
Obviously condition $\mathrm{(2)}$ implies condition $\mathrm{(3)}$. 
We prove the converse implication by indcution on the cardinality of $S$. 
If $\# S\leq 1$, the assertion is trivial.  
If $\#S=2$, 
the assertion follows from Example~\ref{ex:adm squ}. 
Next let us assume that $\# S \geq 3$. 
Then for any distinct elements $s$ and $t\in S$, 
we have the equality
$$(x|^{\{t\}}_{S\ssm\{t\}})|^{\emptyset}_{S\ssm\{s,t\}}=
(x|^{\emptyset}_{S\ssm\{s\}})|^{\{t\}}_{S\ssm\{s,t\}}$$
by the equality $\mathrm{(\ref{equ:restr cube})}$ in \ref{para:rest of cubes} 
and $(x|^{\emptyset}_{S\ssm\{s\}})|^{\{t\}}_{S\ssm\{s,t\}}$ 
is admissible 
by the assumption and the equivalence of 
conditions $\mathrm{(1)}$ and $\mathrm{(2)}$. 
Therefore all frontside faces of $x|_{S\ssm\{t\}}^{\{t\}}$ are admissible. 
On the other hand, 
since $x$ and $x|_{S\ssm\{t\}}^{\emptyset}$ are $0$-spherical, 
$x|_{S\ssm\{t\}}^{\{t\}}$ is also $0$-spherical 
by Lemma~\ref{lem:Cal of Tot}. 
Hence $x|_{S\ssm\{t\}}^{\{t\}}$ is 
admissble for any $t\in S$ by inductive hypothesis. 
The equivalence of conditions $\mathrm{(3)}$ 
and $\mathrm{(4)}$ follows 
from Example~\ref{ex:adm squ} and 
induction on the cardinality of $S$.
\end{proof}

\begin{cor}
\label{cor:comp of cubes}
Let $x$ and $y$ be $S$-cubes in an abelian category $\cA$, 
$S$ an element in $S$ and 
$\alpha:x|_{S\ssm\{s\}}^{\emptyset}\isoto y|_{S\ssm\{s\}}^{\{s\}}$ an isomorphism 
of $S\ssm\{s\}$-cubes. 
If $x$ and $y$ are admissible, 
then $x\circ_{s,\alpha} y$ is also admissible. 
\end{cor}

\begin{proof}[\bf Proof]
We proceed by induction on the cardinality of $S$. 
If $\# S\leq 1$, the assertion is trivial. 
We will check the condition $\mathrm{(3)}$ in Theorem~\ref{thm:char of adm}. 
Since $x$ and $y$ are $0$-spherical, 
$x\circ_{s,\alpha} y$ is also $0$-spherical by Example~\ref{ex:sphericality of comp}. 
Let us assume $\# S\geq 2$. 
We can easily check that 
we have the equality 
$${(x\circ_{s,\alpha}y)}|_T^{\emptyset}=
\begin{cases}
y|_T^{\emptyset} & \text{if $s\notin T$}\\
x|_T^{\emptyset}\circ_{s,\alpha|_T} y|_T^{\emptyset} & \text{if $s\in T$}
\end{cases}$$
for any $T\in\cP(S)$ where 
$\alpha|_T:(x|_T^{\emptyset})|_{T\ssm\{s\}}^{\emptyset} \to 
(y|_T^{\emptyset})|_{T\ssm\{s\}}^{\{s\}}$ 
is a restriction of $\alpha$. 
Therefore $(x\circ_{s,\alpha}y)|_T^{\emptyset}$ is admissible by 
the assumption and the inductive hypothesis. 
Hence $x\circ_{s,\alpha} y$ is admissible. 
\end{proof}

\begin{cor} 
\label{cor:adm cube is fib}
Any admissible $S$-cube in an abelian category $\cA$ is fibered.
\end{cor}

\begin{proof}[\bf Proof] 
The assertion follows from Lemma~\ref{lem:char of fibered S-cube} and 
Example~\ref{ex:adm squ} and Theorem~\ref{thm:char of adm}. 
\end{proof}

\begin{cor}
\label{cor:adm attachment of cubes}
Let $x$ be an admissible $S$-cube and $f:x_{\emptyset} \to y$ 
a monomorphism in $\cA$. 
Then $x_{f,y}$ is also admissible.
\end{cor}

\begin{proof}[\bf Proof]
We prove the condition in 
Theorem~\ref{thm:char of adm} $\mathrm{(3)}$ 
to $x_{f,y}$ by induction on the cardinality of $S$. 
For $\# S=1$, the assertion is trivial and 
for $\# S=2$, the assertion follows from Example~\ref{ex:adm squ} 
and the standard result in the category theory 
Lemma~\ref{lem:attach two} below. 
For $\# S\geq 3$, notice that we have the following equalities 
$$\Homo_k\Tot x_{f,y}=\Homo_k\Tot x=0 \ \ \text{for $k>0$ and}$$
$$x_{f,y}|^{\emptyset}_{S\ssm\{s\}}=
{(x|_{S\ssm\{s\}}^\emptyset)}_{f,y}
\ \ \text{for any $s\in S$}. $$
Hence we get the desired result by inductive hypothesis.
\end{proof}

\begin{lem}
\label{lem:attach two}
For the commutative diagram below in a category, 
$${\footnotesize{\xymatrix{
\bullet \ar[r] \ar[d] 
\ar@{}[rd]|{\bf I}
& 
\bullet \ar@{=}[r]^{\id} \ar[d]_b 
\ar@{}[rd]|{\bf II} 
&
\bullet \ar[d]^{ab}\\
\bullet \ar[r] &
\bullet \ar[r]_a & 
\bullet,}}}$$
if the square {\rm\textbf{I}} is Cartesian and the morphism $a$ 
is a monomorphism, then the square {\rm\textbf{I$+$II}} is also 
Cartesian. 
\qed
\end{lem}

\begin{prop}
\label{prop:char of adm for fib}
Let $\fx=\{x_s\overset{d_s^x}{\rinf} x\}_{s\in S}$ be 
a family of subobjects in $\cA$. 
Then $\Fib \fx$ is an admissible $S$-cube if and only if 
a family $\fx$ is universally admissible in $\cP(\fx)$. 
Here $\cP(\fx)$ is as in Example~\ref{ex:well-powered abelian category}.
\end{prop}

\begin{proof}[\bf Proof]
We proceed by induction on the cardinality of $S$. 
If $\# S\leq 1$, the assertion is trivial. 
If $\# S=2$, then the assertion follows 
from Example~\ref{ex:adm squ}. 
Let us assume $\# S \geq 3$. 
First notice that $\Fib \fx$ is admissible 
if and only if the following three conditions hold by 
Theorem~\ref{thm:char of adm} 
and the equalities $\mathrm{(\ref{equ:face of fib 1})}$ and 
$\mathrm{(\ref{equ:face of fib 2})}$ 
in Example~\ref{ex:ass cubes subobj}.\\
$\mathrm{(1)}$ 
The $S\ssm\{s\}$-cube 
$\Fib(\fx_{S\ssm\{s\}})$ is admissible for any $s\in S$.\\
$\mathrm{(2)}$ 
The $S\ssm\{s\}$-cube 
$\Fib(\fx_{S\ssm\{s\}}\wedge x_s)$ is admissible for any $s\in S$.\\
$\mathrm{(3)}$ 
The $S$-cube $\Fib \fx$ is $0$-spherical.\\
Under conditions $\mathrm{(1)}$ and $\mathrm{(2)}$, 
condition $\mathrm{(3)}$ is equivalent to 
the condition that $\Homo_0\Tot d^{s,x}$ is a monomorphism 
for any $s\in S$ by Lemma~\ref{lem:Cal of Tot} and it 
is equivalent to the condition that 
a family $\fx$ is admissible in $\cP(\fx)$ by Example~\ref{ex:H_0Totds}. 
Therefore by Remark~\ref{rem:univ adm seq}, 
a family $\fx$ is universally admissible in $\cP(\fx)$. 
\end{proof}

\section{Double cubes}
\label{sec:double cubes} 

\sn 
In this section, we 
develop an abstract version of 
the main theorem in Theorem~\ref{thm:big adm implies small adm}.

\begin{df}[\bf Disjoint systems]
\label{df:disjoint fam}
Let $S$ be a set and $T$ a subset of $S$. 
A system of subsets $(A,B,C,D)$ of $S$ 
is a {\bf disjoint system of $S$ with respect to $T$} 
if the following three conditions hold.\\
$\mathrm{(1)}$ 
The sets $A$, $B$, $C$ and $D$ are disjoint in each other.\\
$\mathrm{(2)}$ 
The sets 
$A$ and $B$ are contained in $T$.\\
$\mathrm{(3)}$ 
The sets $C$ and $D$ 
are contained in $S\ssm T$. 
\end{df}

\begin{nt}
\label{nt:DPS}
Let $S$ be a set.\\
$\mathrm{(1)}$ 
For any ordered pair of disjoint subsets $(A,B)$ of $S$, 
we set
$$\DP_{(A,B)}(S):=\{(U,V)\in\DP(S);A\subset U\cup V,\ B\cap V=\emptyset \}$$
and regard it as a partially ordered subset of $\DP(S)$. 
Notice that $\DP_{(\emptyset,\emptyset)}(S)=\DP(S)$. 
For any subset $A$ of $S$, 
we shortly write $\DP_A(S)$ for $\DP_{(A,S\ssm A)}(S)$.\\
$\mathrm{(2)}$ 
Let $T$ be a subset of $S$ and an ordered system 
$(A,B,C,D)$ a disjoint system of $S$ 
with respect to $T$. 
We define 
$i_{(T\subset S),(A,B)}^{(C,D)}:\DP_{(A,B)}(T)\to \DP(S)$ 
to be an order preserving map 
which sends an ordered pair $(U,V)$ in $\DP_{(A,B)}(T)$ to an 
ordered pair $(U\cup C,V\cup D)$ in $\DP(S)$. 
If $T=S$, 
we shortly write $i_{(A,B)}$ for 
$i_{(T\subset S),(A,B)}^{(C,D)}$ and 
it is just the inclusion map $\DP_{(A,B)}(S)\to \DP(S)$.
\end{nt}

\begin{df}[\bf Total functor]
\label{df:Total functor}
For any disjoint pair of subsets $(A,B)$ of $S$, 
we define the {\bf Total functor} ({\bf of $(A,B)$}) 
$\Tot_{(A,B)}^S=\Tot_{(A,B)}:\cP(S) \to \DP_{(A,B)}(S)$ 
by sending a subset $T$ of $S$ to an ordered pair $((A\ssm T)\cup(B\cap T),T\ssm B)$. 
For any subset $A$ of $S$, 
we shortly write $\Tot_A^S=\Tot_A$ for $\Tot_{(A,S\ssm A)}$. 
We also write $e_{\emptyset,\emptyset}^S=e_{\emptyset,\emptyset}$ for 
$\Tot^S_{(\emptyset,\emptyset)}$.
\end{df}

\begin{lem}
\label{lem:fundamental oredered set isom}
For any set $S$ and any subset $A$ of $S$, 
we have an isomorphism of partially ordered sets 
$$\Tot_{A}:\cP(S)\isoto \DP_{A}(S)$$
which sends a subset $T$ to an ordered pair of 
disjoint subsets $(T\ominus A,T\cap A)$ of $S$. 
\end{lem}

\begin{proof}[\bf Proof]
The inverse map of $\Tot_{A}$ is given by sending an ordered pair 
$(U,V)$ in $\DP_{A}(S)$ to a subset $(U\ssm A)\cup V$ of $S$.
\end{proof}

\begin{rem}
\label{nt:e_T} 
Let $T$ be a subset of $S$. 
Recall the definition of $e_T$ from Example~\ref{ex:power set} $\mathrm{(3)}$. 
Then the equality 
$e_T=i_{(T,S\ssm T)}^{(\emptyset,\emptyset)}\Tot_{(T,S\ssm T)}$.
$${\footnotesize{ 
\xymatrix{
\cP(S) \ar[r]_{\!\!\!\!\!\!\!\!\sim}^{\!\!\!\!\!\Tot_{T}} \ar[rd]_{e_T} & 
\DP_{T}(S) \ar[d]^{i_{(T,S\ssm T)}^{(\emptyset,\emptyset)}}\\
& \DP(S).
}}}$$
\end{rem}

\sn 
In the rest of this section, 
let $S$ be a finite set and $\cA$ an abelian category.

\begin{nt}[\bf Restriction of double cubes]
\label{nt:partial doble cube} 
Let $T$ be a subset, $(A,B,C,D)$ a disjoint system of $S$ with respect to $T$ 
and $x$ a double $S$-cube. 
We define the {\bf restriction of $x$ to $(A,B)$ along $(C,D)$} 
by composition of the functors 
$x|_{T,(A,B)}^{(C,D)}:=x{(i_{(T\subset S),(A,B)}^{(C,D)})}^{\op}$. 
$$\DP_{(A,B)}(T)^{\op} \onto{i_{(T\subset S),(A,B)}^{(C,D)}} \DP(S)^{\op} \onto{x} \calD.$$
If $T=S$, we shortly 
write $x|_{(A,B)}^{(C,D)}$ for $x|_{(T\subset S),(A,B)}^{(C,D)}$.
\end{nt}

\begin{lemdf}[\bf Patching of cubes]
\label{lemdf:pat of cubes}
$\mathrm{(1)}$ 
A family $\fx=\{x^T\}_{T\in\cP(S)}$ 
of $S$-cubes in a category $\calD$ 
indexed by the subsets of $S$
is a {\bf patching family} if 
it satisfies the following {\bf patching condition}:
\begin{equation}
\label{equ:pat cond}
x^T|_{S\ssm\{t\}}^{\emptyset}=x^{T\ssm\{t\}}|_{S\ssm\{t\}}^{\{t\}}
\end{equation}
for any subset $T$ of $S$ and any element $t\in T$.\\
$\mathrm{(2)}$ 
Then there exists a unique double $S$-cube $\Pat \fx$ in $\calD$ 
such that 
\begin{equation}
\label{equ:pat cond 2}
e_T^{\ast}\Pat \fx =x^T
\end{equation}
for any subset $T$ of $S$.
\end{lemdf}

\begin{proof}[\bf Proof]
For any $(U,V)$ in $\DP(S)$, there exists a pair of 
subsets $T$ and $W$ of $S$ such that 
$$(U,V)=e_T(W)$$
by the equality (\ref{equ:eUV 2}) in Notations~\ref{nt:DPS}. 
We put ${(\Pat \fx)}_{U,V}:=x^T_W$. 
We need to check that this definition does not depend upon the choice 
of subsets $T$ and $W$ of $S$. 
By virtue of the equality (\ref{equ:eUV 2}) again, 
we shall assume that $W\neq S$ and $T\neq \emptyset$ 
and we just need to verify the equality 
$x_W^T=x^{T\ssm\{t\}}_{W\coprod\{t\}}$ 
for any $t\in T$. 
Since $W\subset S\ssm\{t\}$, 
we have the equality 
$x_W^T={(x^T|_{S\ssm\{t\}}^{\emptyset})}_W=
{(x^{T\ssm\{t\}}|^{\{t\}}_{S\ssm\{t\}})}_W=x^{T\ssm\{t\}}_{W\coprod\{t\}} $ 
by the equality (\ref{equ:pat cond}). 
Hence we get the well-definedness of ${(\Pat \fx)}_{U,V}$. 
The definition of boundary morphisms of $\Pat \fx$ is similar. 
\end{proof}

\begin{ex}[\bf Reconstruction of double cubes]
\label{ex:reconst dc}
Let $x$ be a double $S$-cube in a category $\calD$. 
Then the family $\fx=\{e_T^{\ast}x\}_{T\in\cP(S)}$ is a patching family 
and $\Pat \fx=x$. 
\end{ex}

\begin{thm}
\label{thm:big adm implies small adm}
Let $x$ be a double $S$-cube in $\cA$. 
We assume that the following four conditions hold.\\
$\mathrm{(1)}$ 
The $S$-cube $\mathbf{2}^{\ast}x$ is admissible.\\
$\mathrm{(2)}$ 
The double $S$-cube $x$ is monic.\\
$\mathrm{(3)}$ 
The $S$-cube $e_T^{\ast}x$ is fibered 
for any proper subset $T$ of $S$.\\
$\mathrm{(4)}$ 
For any subset $W$ of $S$, 
any pair of non-empty disjoint subsets $U$ and $V$ of $S$ 
such that $W\cup U\neq S$ 
and 
any element $s\in S\ssm (W\cup U)$ and $v\in V$ 
such that $s\neq v$, 
we put $V':=V\ssm\{v\}$ and 
a family of morphisms to $x_{W\ssm V',V'}$ 
$${\tiny{\fx:=\{x_{W\ssm(V\coprod\{k\}),V\coprod\{k\}} 
\onto{d_k^{\fx}}  
x_{W\ssm V',V'}\}_{k\in U}\coprod
\{x_{(W\coprod\{s\})\ssm V',V'} 
\onto{d_s^{\fx}} 
x_{W\ssm V',V'}\}}}$$
where $d_k^{\fx}=x((W\ssm V',V')\leq(W\ssm(V\coprod\{k\}),V\coprod\{k\}))$ 
for any $k\in U$ and 
$d_s^{\fx}=d^{s,x}_{(W\coprod\{s\})\ssm V',V'}$.
Then the morphism 
$$\Homo_0\Tot d^{s,\Fib\fx}:\Homo_0\Tot\Fib\fx|_U^{\{s\}} 
\to \Homo_0\Tot\Fib \fx|_U^{\emptyset}$$
is a monomorphism.\\
Then the $S$-cube $e_S^{\ast}x$ is also an admissible $S$-cube.
\end{thm}

\sn
To prove the theorem, 
we utilize the following lemma 
which is a standard result in the category theory.

\begin{lem}
\label{lem:cat fund}
Let $\calD$ be a category. 
Then\\
$\mathrm{(1)}$ 
For a pair of composable morphisms 
$\bullet \onto{f} \bullet \onto{g} \bullet$ in $\calD$, 
if $gf$ is a monomorphism, 
then $f$ is also a monomorphism.\\
$\mathrm{(2)}$ 
In the commutative diagram in $\calD$ below, 
$${\footnotesize{\xymatrix{
\bullet \ar[r] \ar[d] \ar@{}[dr]|{\rm{\bf{I}}} & 
\bullet \ar[r]^b \ar[d] \ar@{}[dr]|{\rm{\bf{II}}} & 
\bullet \ar[d]\\
\bullet \ar[r] \ar[d]_a \ar@{}[dr]|{\rm{\bf{III}}} & 
\bullet \ar[r] \ar[d] \ar@{}[dr]|{\rm{\bf{IV}}} & 
\bullet \ar[d]\\
\bullet \ar[r] & \bullet \ar[r] & \bullet,
}}}$$
if the big square 
{\rm{\bf{I}}}$+${\rm{\bf{II}}}$+${\rm{\bf{III}}}$+${\rm{\bf{IV}}} is 
a Cartesian square and 
if the morphisms $a$ and $b$ in the diagram above are monomorphisms, 
then the squares {\rm{\bf{I}}}$+${\rm{\bf{II}}}, 
{\rm{\bf{I}}}$+${\rm{\bf{III}}} and {\rm{\bf{I}}} are Cartesian.
\qed
\end{lem}

\begin{proof}[\bf Proof of Theorem~\ref{thm:big adm implies small adm}]
If $\# S\leq 2$, the assertion follows from Lemma~\ref{lem:cat fund}. 
We assume that $\# S \geq 3$. 
We put $\fa=\{x_{\{s\},\emptyset}\}_{s\in S}$ and 
$\fb=\{x_{\emptyset,\{s\}} \}_{s\in S}$. 
We have the canonical isomorphism 
$x_{U,V}\isoto \fa^{\wedge U}\wedge \fb^{\wedge V}$ 
as subobjects of $x_{\emptyset,\emptyset}$ 
by Corollary~\ref{cor:fib discribe}. 
Then conditions $\mathrm{(1)}$ and $\mathrm{(4)}$ 
are equivalent to the following conditions 
$\mathrm{(1)'}$ and $\mathrm{(4)'}$ respectively 
by Proposition~\ref{prop:char of adm for fib} 
and Example~\ref{ex:H_0Totds} respectively.\\
$\mathrm{(1)'}$ The family $\fb$ of subobjects in 
$x_{\emptyset,\emptyset}$ is universally admissible in 
a lattice $\cP(x)$.\\
$\mathrm{(4)'}$ For any subset $W$ of $S$, 
any pair of non-empty disjoint subsets $U$ and $V$ of $S$ 
such that $W\cup U\neq S$ and 
any elements $s\in S\ssm (W\cup U)$ and $v\in V$ 
such that $s\neq v$, 
a pair 
$(\fb_{U}\wedge\fb^{\wedge V}\wedge \fa^{\wedge W},
\fb^{\wedge V\ssm\{v\}}\wedge\fa^{\wedge W\coprod\{s\}})$ is distributive.\\
The family of subobjects 
$\fz:=\{x_{S\ssm T,T} \rinf x_{S,\emptyset} \}_{T\in\cP(S)}$ 
of $x_{S,\emptyset}$ 
is universally admissible in $\cP(\fz)$ by 
Example~\ref{ex:well-powered abelian category}, 
Corollary~\ref{cor:adm seq cor} and 
Remark~\ref{rem:adm seq cor rem}. 
Hence $e_{S}^{\ast}x\isoto \Fib \fz$ is admissible 
by Proposition~\ref{prop:char of adm for fib} again.
\end{proof}

\begin{para}
\label{para:proof dct}
\begin{proof}[\bf Proof of Theorem~\ref{thm:dct}]
If $\# S\leq 2$, then the assertion follows from Lemma~\ref{lem:cat fund}. 
If $\# S \geq 3$, then we will prove condition $\mathrm{(3)}$ implies 
conditions $\mathrm{(3)}$ and $\mathrm{(4)}$ in 
Theorem~\ref{thm:big adm implies small adm}. 
Condition $\mathrm{(3)}$ follows 
from Corollary~\ref{cor:adm cube is fib}. 
Inspection shows that $\Fib \fx$ in 
condition $\mathrm{(3)}$ in Theorem~\ref{thm:big adm implies small adm} 
is written by 
compositions of 
restrictions of faces of $e_T^{\ast}x$ 
for some proper subsets $T$ of $S$. 
More precisely, we put 
$U_1:=\{s,v\}\coprod U$ and $U_2:=U\coprod\{s\}$ and 
$$\fx_1:=\{x_{(W\ssm V')\cup\{k\},V'}\to x_{W\ssm V',V'} \}_{k\in U_1},$$
$$\fx_2:=\{x_{W\ssm V,V}\to x_{(W\ssm V')\cup\{v\},V'} \}\coprod\{x_{(W\ssm V')\cup\{k,v\},V'} \to x_{(W\ssm V')\cup\{v\},V'} \}_{k\in U_2}.$$ 
Then $\Fib x$ can be written by the composition of 
$$\Fib \fx_1=\Tot_{(\emptyset,U_1)}x|^{(W\ssm V',V')}_{U_1,(\emptyset,U_1)}$$ 
and 
$$\Fib \fx_2=\Tot_{(\{v\},U_2)}x|_{U_1,(\{v\},U_2)}^{(W\ssm V',V')}$$ 
and $\Tot^{U_2}_{(\emptyset,\{s\})}x|_{U_2,(\emptyset,\{s\})}^{(W\ssm V,V)}$. 
Therefore $\Fib \fx$ is admissible by Corollary~\ref{cor:comp of cubes}. 
Hence we obtain the desired result. 
\end{proof}
\end{para}

\section{Regular adjugates of cubes}
\label{sec:reg adj cubes}

In this section, 
let $A$ be a commutative ring with unit and 
we study the notion about (regular) adjugates of cubes 
in an $A$-linear abelian category.

\begin{nt}
\label{nt:A-lincat}
An {\bf $A$-linear category} $\cC$ 
is a category enriched over the 
category of $A$-modules. 
Namely for any pair of objects $x$ and $y$ in $\cC$, 
the set of morphisms from $x$ to $y$, $\Hom_{\cC}(x,y)$ has 
a structure of $A$-module and 
the composition of morphisms 
$\Hom_{\cC}(x,y)\otimes_A \Hom_{\cC}(y,z) \to \Hom_{\cC}(x,z)$ 
is a homomorphism of $A$-modules 
for any objects $x$, $y$ and $z$ in $\cC$. 
We fix an $A$-linear category $\cC$. 
For an element $a$ in $A$ and an object $x$ in $\cC$, 
we write $a_x$ for the morphism $a\id_x:x \to x$. 
In particular, we notice that we have the equality 
$fa_x=a_yf$ 
for any morphism $f:x\to y$ in $\cC$ and 
an element $a$ in $A$. 
The category of $A$-modules is 
a typical example of $A$-linear category.
\end{nt}

\begin{lemdf}[\bf Adjugate]
\label{lemdf:adjugate}
Let $f:x\to y$ and $\phi:y' \to y$ be morphisms in $\cC$ and 
let us assume that there exists an element $a$ in $A$ 
and a morphism $f^{\ast}:y \to x$ such that 
we have the equalities 
$f^{\ast}f=a_x$ and $ff^{\ast}=a_y$. 
We call a pair $(f^{\ast},a)$ or 
simply $f^{\ast}$ an {\bf adjugate of $f$}. 
Moreover let us assume that the diagram below is the pull-back 
of $f$ along $\phi$. 
$${\footnotesize{\xymatrix{
x' \ar[r]^{f'} \ar[d]_{\phi'} & y' \ar[d]^{\phi}\\
x \ar[r]_f & y \ .
}}}$$
Then there exists a unique morphism ${f'}^{\ast}:y' \to x'$ 
which satisfies the following equalities.\\
$\mathrm{(1)}$ 
${f'}^{\ast}f'=a_{y'}$ and $f'{f'}^{\ast}=a_{x'}$.\\
$\mathrm{(2)}$ 
$\phi' {f'}^{\ast}=f^{\ast}\phi$. 
Namely the following diagram is commutative. 
$${\footnotesize{\xymatrix{
y' \ar[r]^{{f'}^{\ast}} \ar[d]_{\phi} & x' \ar[d]^{\phi'}\\
y \ar[r]_{f^{\ast}} & x \ .
}}}$$
We call the morphism ${f'}^{\ast}$ 
the {\bf adjugate of $f'$ induced from $f^{\ast}$} 
({\bf along $\phi$}).
\end{lemdf}

\begin{proof}[\bf Proof] 
Since we have the equality 
$ff^{\ast}\phi=a_y\phi=\phi a_{y'}$, 
there exists, 
by the universal property of fiber product,  
a unique morphism 
${f'}^{\ast}:y'\to x'$ 
with which we can fill in the dotted arrow in 
the following commutative diagram:
$${\footnotesize{\xymatrix{
x'\ar[r]^{f'} \ar[d]_{\phi'} \ar@/^3pc/[rrr]^{f'a_{x'}}& y' 
\ar@{-->}[r]^{{f'}^{\ast}} \ar[d]_{\phi} \ar@/^2pc/[rr]_{a_{y'}}& 
x' \ar[r]^{f'} \ar[d]^{\phi'} & y' \ar[d]^{\phi}\\
x \ar[r]_f & y\ar[r]_{f^{\ast}} & x \ar[r]_{f} & y \ .  
}}}$$
Applying the two morphisms 
$a_{x'}:x' \to x'$ and ${f'}^{\ast}f':x'\to x'$ 
to the universal property of $x'$ again, 
we acquire the equality ${f'}^{\ast}f'=a_{x'}$.
\end{proof}

\begin{ex}[\bf Adjugate of matrices]
\label{ex:adj of mat}
Let $X$ be an $n\times n$ matrix whose coefficients are in $A$. 
We regard $X$ as a homomorphism of free $A$-modules $X:A^{\oplus n} \to A^{\oplus n}$. 
We denote the {\bf adjugate} of $X$ by $\adj X$. 
Namely the matrix $\adj X$ is an $n\times n$ matrix 
whose $(i,j)$-entry is given by $(-1)^{i+j}\det X_{j,i}$ where 
$X_{j,i}$ is the $(j,i)$-cofactor of $X$ and 
$\det X_{j,i}$ means the {\bf determinant} of $X_{j,i}$. 
It is well-known that we have the equality 
$(\adj X) X=X\adj X=(\det X) E_n$ 
where $E_n$ is the $n$-th unit matrix. 
Then a pair $(\adj X, \det X)$ is an adjugate of $X$. 
\end{ex}

\begin{ex}[\bf Typical cubes]
\label{ex:abs typ cubes}
Let $\ff_S=\{f_s\}_{s\in S}$ a family of elements in $A$ and $x$ an object in $\cC$. 
We define $\Typ(\ff_S;x)$ an $S$-cube in $\cC$ called the {\bf typical cubes associated with 
$\ff_S$ and $x$} as follows. 
For any $T\in\cP(S)$ and any element $t$ in $T$, 
we put $\Typ(\ff_S;x)_T=x$ and $d_T^{t,\Typ(\ff_S;x)}:={(f_t)}_x$. 
For example, 
if $\cC$ is the category of $A$-modules and $x=A$, 
then $\Typ(\ff_S;A)$ is just the typical cube associated with $\ff_S$ 
as in Example~\ref{ex:mot ex}. 
\end{ex}

\begin{nt}[\bf $x$-regular sequences and $x$-sequences]
\label{nt:regular sequences}
Let us assume that $\cC$ is an additive category and moreover 
for any finite family of morphisms $\{\phi_i:y_i \to x \}_{1\leq i\leq q}$ in $\cC$, 
there exists the cokernel 
$x/(\phi_1,\cdots,\phi_q):=\coker (\oplus y_i \onto{\oplus \phi_i} x)$ in $\cC$. 
For any elements $f_1,\cdots,f_q$ in $A$ and an object $x$ in $\cC$, 
we simply write $x/(f_1,\cdots,f_q)$ for $x/({(f_1)}_x,\cdots,{(f_q)}_x)$. 
Let us fix an object $x$ in $\cC$.\\
$\mathrm{(1)}$ 
A sequence of elements $f_1,\cdots,f_q$ in $A$ is 
an {\bf $x$-regular sequence} if every $f_i$ is a non-unit in $A$, 
if ${(f_1)}_x$ is a monomorphism in $\cC$ and if 
${(f_{i+1})}_{x/(f_1,\cdots,f_i)}$ is 
a monomorphism for any $1\leq i\leq q-1$.\\
$\mathrm{(2)}$ 
A finite family $\{f_s\}_{s\in S}$ of elements in $A$ is an {\bf $x$-sequence} 
if $\{f_s\}_{s\in S}$ forms an $x$-regular sequence 
with respect to every ordering of the members of $\{f_s\}_{s\in S}$. 
\end{nt}

\begin{lem}
\label{lem:char of x-seq}
Let $\ff_S=\{f_s\}_{s\in S}$ be a family of elements in $A$ and 
$x$ an object in $\cC$ and we put 
$\ff_{x,S}:=\{{(f_s)}_x:x\to x \}_{s\in S}$ a family of morphisms in $\cC$. 
Then the following conditions are equivalent.\\
$\mathrm{(1)}$ 
The $S$-cube $\Typ(\ff_S;x)$ is admissible.\\
$\mathrm{(2)}$ 
The family $\ff_S$ is an $x$-sequence.\\
$\mathrm{(3)}$ 
The morphism ${(f_s)}_x:x\to x$ is a monomorphism for any $s\in S$ and 
the family $\ff_{x,S}$ is admissible in $\cP(\ff_{x,S})$.\\
$\mathrm{(4)}$ 
The morphism ${(f_s)}_x:x\to x$ is a monomorphism for any $s\in S$ and 
the family $\ff_{x,S}$ is universally admissible in $\cP(\ff_{x,S})$. 
\end{lem}

\begin{proof}[\bf Proof] 
We prove that condition $\mathrm{(1)}$ implies 
condition $\mathrm{(4)}$. 
If $\Typ(\ff_S;x)$ is admissible, then 
$\Typ(\ff_S;x)\isoto\Fib \ff_{x,S}$ by 
Corollary~\ref{cor:adm cube is fib}. 
Therefore a family $\ff_{x,S}$ is universally admissible by 
Proposition~\ref{prop:char of adm for fib}. 
Obviously condition $\mathrm{(4)}$ implies condition $\mathrm{(3)}$. 
Since for any non-empty subset $T$ of $S$ and any element $s\in S\ssm T$, 
we have the isomorphism 
$$\Ker({(f_s)}_{x/\ff_{T}}:x/\ff_{T} \to x/\ff_{T})\isoto 
\frac{\ff^{\vee T}\wedge {(f_s)}_x}{{(\ff_S\wedge {(f_s)}_x)}^{\vee T}}$$ 
by Example~\ref{ex:H_0Totds} 
where ${(f_s)}_x$ means the subobject ${(f_s)}_x:x\to x$ of $x$,  
We can easily notice that 
condition $\mathrm{(3)}$ implies 
condition $\mathrm{(2)}$. 
We prove condition $\mathrm{(2)}$ implies
 condition $\mathrm{(1)}$ by 
induction on the cardinality of $S$. 
If $\# S\leq 1$,
the assertion is trivial. 
For $\# S\geq 2$, 
we first notice that we have the equalities
\begin{equation}
\label{equ:Typ rest}
\Typ(\ff_S;x)|_{S\ssm\{s\}}^{\emptyset}=\Typ(\ff_S;x)|_{S\ssm\{s\}}^{\{s\}}=\Typ(\ff_{S\ssm\{s\}};x)
\end{equation}
for any $s$ in $S$. 
Therefore all faces of $\Typ(\ff_S;x)$ are admissible 
by the inductive hypothesis and Example~\ref{ex:adm squ}. 
Moreover by Lemma~\ref{lem:Cal of Tot}, we have the isomorphisms
$$\Homo_p(\Tot\Typ(\ff_S;x))\isoto 
\begin{cases}
x/\ff_S & \text{if $p=0$}\\
\frac{\ff^{\vee S\ssm\{s\}}\wedge {(f_s)}_x}
{{(\ff_{S\ssm\{s\}}\wedge {(f_s)}_x)}^{\vee S\ssm\{s\}}}=0 & \text{if $p=1$}\\ 
\Homo_{p-1}(\Tot\Typ(\ff_{S\ssm\{s\}};x))=0 & \text{if $p\geq 2 $}.
\end{cases}$$
Therefore $\Typ(\ff_S;x)$ is $0$-spherical and hence $\Typ(\ff_S;x)$ is 
admissible by Example~\ref{ex:H_0Totds} and Theorem~\ref{thm:char of adm}.
\end{proof}

\begin{df}[\bf Adjugate of cubes]
\label{df:adjugate of cubes}
$\mathrm{(1)}$ 
An {\bf adjugate of an $S$-cube} $x$ in $\cC$ is 
a pair $\fA=(\fa,\fd^{\ast})$ 
consisting of a family of elements $\fa=\{a_s\}_{s\in S}$ in $A$ 
and a family of morphisms 
$\fd^{\ast}=\{d_T^{t \ast}:x_{T\ssm\{t\}} \to x_T\}_{T\in\cP(S),t\in T}$ 
in $\cC$ which satisfies the following two conditions.\\
$\mathrm{(i)}$ 
We have the equalities 
$d_T^td_T^{t \ast}={(a_t)}_{x_{T\ssm\{t\}}}$ 
and $d_T^{t \ast}d_T^t={(a_t)}_{x_T}$ 
for any $T\in\cP(S)$ and $t\in T$.\\
$\mathrm{(ii)}$ 
For any $T\in\cP(S)$ and any two distinct elements $a$ and $b\in T$, 
we have the equality 
$d_T^bd_T^{a \ast}=d^{a \ast}_{T\ssm\{b\}}d^b_{T\ssm\{a\}}$. 
Namely, the following diagram is commutative. 
$${\footnotesize{\xymatrix{
x_{T\ssm\{a\}} \ar[r]^{d_T^{a \ast}} \ar[d]_{d^b_{T\ssm\{a\}}} & x_T \ar[d]^{d_T^b}\\
x_{T\ssm\{a,b\}} \ar[r]_{d_{T\ssm\{b\}}^{a \ast}} & x_{T\ssm\{b\}}\ .
}}}$$
$\mathrm{(2)}$ 
An adjugate of an $S$-cube $\fA=(\fa,\fd^{\ast})$ is {\bf regular} 
if $\fa$ forms an $x_T$-sequence for any $T\in \cP(S)$.
\end{df}

\begin{ex}
\label{ex:adj cube}
The notion of an adjugate of an $S$-cube is related to the notion of
the adjugate of an matrix in the following sense.
Let $x$ be an $S$-cube and suppose that,
for any $T \in \cP(S)$, the $A$-module $x_T$ is free of finite rank
and its rank is independent of $T$. 
For each $T$ let us fix a basis of $x_T$ over $A$ and 
for each pair $(T,t)$ of $T \in \cP(S)$ and $t \in T$, let
$X_T^t$ denote the matrix for the map 
$d^t_T\colon x_T \to x_{T \ssm \{t\}}$ with respect
to the fixed basis of $x_T$ and $x_{T \ssm \{t\}}$.
For each $s \in S$, let us fix an element $a_s \in A$ which is
a common multiple of the determinants of $X_T^s$, where
$T$ runs over the subsets of $S$ containing $s$.
For each $T \in \cP(S)$ containing $s$, let us choose
$b_T^s \in A$ satisfying $a_s = b_T^s \det X_T^s$.
Let $d_T^{s*} \colon x_{T \ssm \{s\}} \to x_T$ denote
the homomorphism of $A$-modules whose matrix with
respect to the fixed basis is given by the adjugate 
of $X_T^s$ multiplied by $b_T^s$. Then the pair
$(\{a_s \}_{s \in S}, \{d_T^{t *}\}_{T \in \cP(S), t \in T})$
is an adjugate of the $S$-cube $x$.
\end{ex}

\begin{rem}
\label{rem:xast}
Let $x$ be an $S$-cube in $\cC$ and 
$\fA=(\fa=\{a_s\}_{s\in S},
\fd^{\ast}=\{d_T^{t\ast}:x_{T\ssm\{t\}}\to x_T\}_{T\in\cP(S),t\in T})$ 
an adjugate of $x$. 
For any subset $T$ of $S$ such that $\# T\geq 2$ 
and any pair of distinct elements $s$ and $t$ in $T$, 
if $d_T^{s,x}$ or $d_T^{t,x}$ is a monomorphism, 
then the following diagram is commutative. 
$${\footnotesize{\xymatrix{
x_{T\ssm\{s,t\}} \ar[r]^{d_{T\ssm\{s\}}^{t\ast}} 
\ar[d]_{d_{T\ssm\{t\}}^{s\ast}} & 
x_{T\ssm\{s\}} \ar[d]^{d_T^{s\ast}}\\
x_{T\ssm\{t\}} \ar[r]_{d_T^{t\ast}} & x_T . 
}}}$$
In particular, 
if $x$ is monic, 
then we can define the $S$-cocube $x^{\ast}$ by 
$x^{\ast}_T=x_T$ and $d_T^{t,x^{\ast}}:=d_T^{t\ast}$. 
We call $x^{\ast}$ the {\bf adjugate $S$-cocube} 
({\bf associated with an adjugate 
$\fA=(\fa,\fd^{\ast})$}). 
In this case, notice that a pair 
$\fA^{\ast}:=(\fa=\{a_s\}_{s\in S},
\fd:={d_T^{t,x}:x_T \to x_{T\ssm\{t\}}}_{T\in\cP(S),t\in T})$ 
is an adjugate of an $S$-cube $\widehat{x^{\ast}}$. 
We call $\fA^{\ast}$ the {\bf dual adjugate of $\fA$}. 
Obviously we have the equality $\widehat{{(\widehat{x^{\ast}})}^{\ast}}=x$. 
\end{rem}

\noindent
We use the following notation. 
For any pair of sets $U$ and $V$, 
we put 
$U\ominus V:=(U\cup V)\ssm (U\cap V)$.

\begin{ex}[\bf Patching families associated to adjugates]
\label{ex:pat fam ass adj}
Let $x$ be a monic $S$-cube in $\cC$ and 
$\fA=(\fa=\{a_s\}_{s\in S},
\fd^{\ast}=\{d_T^{t\ast}:x_{T\ssm\{t\}}\to x_T\}_{T\in\cP(S),t\in T})$ 
an adjugate of $x$. 
We construct the {\bf patching family 
$\fP_{\fA}:=\{x^{\fA,T}\}_{T\in \cP(S)}$ 
associated to an adjugate $\fA$} 
as follows. 
For $T\in \cP(S)$, 
we define the $S$-cube $x^{\fA,T}$ by setting 
$x^{\fA,T}_U:=x_{U\ominus T}$ for any $U\in \cP(S)$ 
and for any $u\in U$, 
$d_U^{u,x^{\fA,T}}:=d_{U\ominus T}^{u,x}$ if $u\in U\ssm T$ and 
if $d_U^{u,x^{\fA,T}}:=d_{(U\ominus T)\coprod\{u\}}^{u\ast}$ if $u\in U\cap T$. 
As in Remark~\ref{rem:xast}, 
we can easily check that  $x^{\fA,T}$ is an $S$-cube for any $T\in\cP(S)$. 
For any subsets $U$ and $T$ of $S$ and any element $u$ in $U$, 
we define a morphism 
$d_U^{u,\fA^T}:x^{\fA,T}_{U\ssm\{u\}}\to x^{\fA,T}_U$ 
by $d_U^{u,\fA^T}=d^{u\ast}_{U\ominus T}$ if $u\in U\ssm T$ and 
$d_U^{u,\fA^T}=d^{u,x}_{(U\ominus T)\coprod\{u\}}$ if $u\in U\cap T$. 
Then a pair 
$\fA^T=(\fa=\{a_s\}_{s\in S},\fd^{\fA,T}=
\{d_U^{u,\fA^T}:x_{U\ssm\{u\}}^{\fA,T} \to 
x_{U}^{\fA,T} \}_{U\in\cP(S),u\in U})$ is an adjugate of $x^{\fA,T}$. 
Notice that $x^{\fA,\emptyset}=x$ and $\fA^{\emptyset}=\fA$ and 
$x^{\fA,S}=\widehat{x^{\ast}}$ and $\fA^{S}=\fA^{\ast}$. 
We check the patching condition for $\fP_{\fA}$. 
For any $T\in\cP(S)$ and any $t\in T$, 
we need to check the equality 
$x^{\fA,T}|_{S\ssm\{t\}}^\emptyset=
x^{\fA,T\ssm\{t\}}|_{S\ssm\{t\}}^{\{t\}}$. 
We fix a subset $W$ of $S\ssm\{t\}$ and an element $w$ in $W$. 
Then we have the equalities. 
$${(x^{\fA,T}|_{S\ssm\{t\}}^\emptyset)}_W=
x_{W\ominus T}=x_{(W\coprod\{t\})\ominus(T\ssm\{t\})}=
{(x^{\fA,T\ssm\{t\}}|_{S\ssm\{t\}}^{\{t\}})}_W \ \ \text{and}$$
\begin{multline*}
d_W^{w,x^{\fA,T}|_{S\ssm\{t\}}^\emptyset}= 
\begin{cases}
d_{T\ominus W}^{w,x} & \text{if $w\in W\ssm T$}\\
d_{(T\ominus W)\coprod\{w\}}^{w\ast} & \text{if $w\ \in W\cap T$}
\end{cases}\\
=
\begin{cases}
d_{(T\ssm\{t\})\ominus(W\coprod\{t\})}^{w,x} & \text{if $w\in W\ssm T$}\\
d_{\{(T\ssm\{t\})\ominus(W\coprod\{t\})\}\coprod\{w\}}^{w\ast} & \text{if $w\ \in W\cap T$}
\end{cases}
=d_W^{w,x^{\fA,T\ssm\{t\}}|_{S\ssm\{t\}}^{\{t\}}}.
\end{multline*} 
Hence $\fP_{\fA}$ is a patching family of $S$-cubes in $\cC$. 
\end{ex}

\begin{df}[\bf Restriction of adjugates]
\label{df:res of adj}
Let $x$ be an $S$-cube in $\cC$ and 
$\fA=(\fa:=\{a_s\}_{s\in S},\fd^{\ast}:=\{d^{t\ast}_T:x_{T\ssm\{t\}} \to x_T \}_{T\in\cP(S),\ t\in T})$ 
an adjugate of $x$. 
For any disjoint pair of subsets $U$ and $V$ of $S$, 
we define an adjugate $\fA|_U^V:=(\fa|_U^V,\fd^{\ast}|_U^V)$ as follows. 
We put $\fa|_U^V:=\{a_u\}_{u\in U}$ and 
$\fd^{\ast}|_U^V:=\{d^{t\ast}_{V\coprod T}:x_{V\coprod(T\ssm\{t\})} \to x_{V\coprod T} \}_{T\in\cP(U),\ t\in T}$. 
We call $\fA|_U^V$ the {\bf restriction of $\fA$} 
({\bf to $U$ along $V$}).
\end{df}

\begin{lem}[\bf Compatibility of restriction and patching families]
\label{lem:comp of rest and pat}
Let $x$ be a monic $S$-cube in $\cC$. 
Then for any disjoint pair of subsets $U$ and $V$ 
of $S$ and any subset $T$ of $S$, 
if we suppose either the condition $\mathrm{(1)}$ 
or the condition $\mathrm{(2)}$ below, 
then we have the equality
\begin{equation}
\label{equ:comp rest and pat}
x^{\fA,T}|_U^V={(x|_U^V)}^{\fA|_U^V,T\cap U}.
\end{equation}
$\mathrm{(1)}$ 
$T\subset U$.\\
$\mathrm{(2)}$ 
$V\subset T$ and $V=S\ssm U$.
\end{lem}

\begin{proof}[\bf Proof] 
We suppose either condition $\mathrm{(1)}$ or condition $\mathrm{(2)}$. 
For any subset $W$ of $S$ and any element $w$ of $W$, 
we have the equalities
$${\{(x^{\fA,T})|_U^V\}}_W=
x^{\fA,T}_{W\coprod V}=
x_{(W\coprod V)\ominus T}=
x_{\{W\ominus (T\ssm V)\}\coprod V}=
{(x|_U^V)}_{W\ominus (T\ssm V)}=
{\{{(x|_U^V)}^{\fA|_U^V,T\cap U}\}}_W \ \ \text{and}$$
\begin{multline*}
d_W^{w,x^{\fA,T}|_U^V}=d_{W\coprod V}^{w,x^{\fA,T}}=
\begin{cases}
d^{w,x}_{(W\coprod V)\ominus T} & \text{if $w\in (W\coprod V)\ssm T $}\\
d^{w\ast}_{(W\coprod V)\ominus T\coprod\{w\}} & \text{if $w\in (W\coprod V)\cap T$}
\end{cases}\\
=
\begin{cases}
d_{W\ominus (T\ssm V)}^{w,x|_U^V} & \text{if $w\in W\ssm (T\ssm V)$}\\
d_{(W\ominus (T\ssm V))\coprod\{w\}}^{w,\fA|_U^V} & \text{if $w\in W\cap (T\ssm V)$}
\end{cases}
=d_W^{w,{(x|_U^V)}^{\fA|_U^V,T\cap U}}.
\end{multline*}
Hence we have the equality $\mathrm{(\ref{equ:comp rest and pat})}$.
\end{proof}

\sn 
The following result is an easy corollary of 
Lemma-Definition~\ref{lemdf:adjugate}.

\begin{cor}
\label{cor:adjugate}
Let $\fx=\{x_s \overset{d_s^x}{\rinf} x\}_{s\in S}$ be a family of subobjects 
in an $A$-linear abelian category. 
If all $d_s^x$ admit non-trivial adjugates 
$(d_s^{x,\ast}:x \to x_s,a_s)$, 
then $\Fib \fx$ also admits the adjugate 
$(\{a_s\}_{s\in S},\{d_T^{t \ast}:x_{T\ssm\{t\}} \to x_T \}_{T\in \cP(S),\ t\in T})$ 
such that 
$d_{\emptyset}^{s \ast}=d_s^{x,\ast}$ for any $s$ in $S$ 
and $d^{t \ast}_T$ is the adjugate 
of $d_T^t$ induced from $d^t_{T\ssm \{s\}}$ along $d^s_{T\ssm\{t\}}$ 
for any subset $T$ of $S$ with $\# T \geq 2$ 
and for any distinct pair of elements $s$ and $t$ in $T$.
\end{cor}

\begin{proof}[\bf Proof] 
Let $(\{d_s^{x \ast}:x\to x_s\}_{s\in S}, \{a_s\}_{s\in S})$ be 
a family of adjugates of $\fx$. 
Namely $\{a_s\}_{s\in S}$ is a family of elements in $A$ and 
we have equalities $d_s^{x \ast}d_s^x={(a_s)}_{x_s}$ and 
$d_s^xd_s^{x \ast}={(a_s)}_x$ for any $s\in S$. 
We inductively define a morphism 
$d_T^{t \ast}:x_{T\ssm\{t\}} \to x_T$ 
for any non-empty subset $T$ of $S$ and any element $t$ in $T$. 
For $\# T\geq 2$, we fix an element $s$ in $T$ and 
let $d_T^{t \ast}$ be the adjugate of $d_T^t$ induced from $d_{T\ssm\{s\}}^t$ along $d_{T\ssm\{t\}}^s$. 
If $\# T\geq 3$, for any element $u\in T\ssm\{s,t\}$, 
we consider the following diagram.
$${\footnotesize
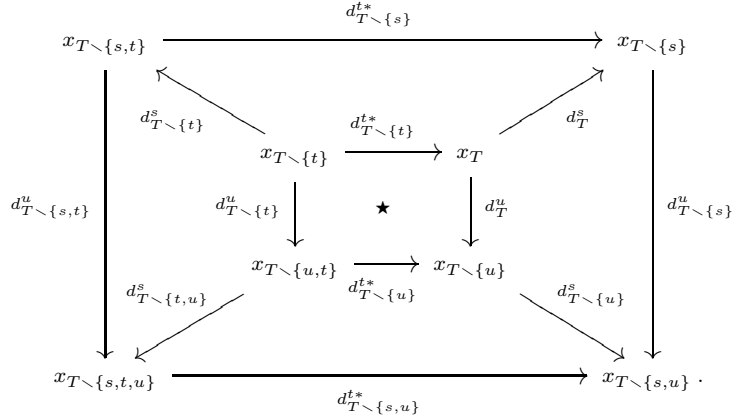
{\xymatrix{
x_{T\ssm\{s,t\}} \ar[rrr]^{d_{T\ssm\{s\}}^{t \ast}} \ar[ddd]_{d_{T\ssm\{s,t\}}^u} & & & 
x_{T\ssm\{s\}} \ar[ddd]^{d_{T\ssm\{s\}}^u}\\
& x_{T\ssm\{t\}} \ar[r]^{d_{T\ssm\{t\}}^{t \ast}} \ar[d]_{d_{T\ssm\{t\}}^u} 
\ar[ul]^{d_{T\ssm\{t\}}^s} \ar@{}[rd]|\bigstar
& x_{T} \ar[d]^{d_T^u} \ar[ur]_{d_T^s} & \\
& x_{T\ssm\{u,t\}} \ar[r]_{d_{T\ssm\{u\}}^{t \ast}} \ar[dl]_{d_{T\ssm\{t,u\}}^s} 
& x_{T\ssm\{u\}} \ar[dr]^{d_{T\ssm\{u\}}^s} & \\
x_{T\ssm\{s,t,u\}} \ar[rrr]_{d_{T\ssm\{s,u\}}^{t \ast}} & & & 
x_{T\ssm\{s,u\}}\ .
}}}$$
Since $\Fib x$ is an $S$-cube and by the inductive hypothesis, 
all squares except $\bigstar$ in the diagram above are commutative. 
Notice that the morphism $d_{T\ssm\{u\}}^{s}$ 
is a monomorphism and hence the square $\bigstar$ is also commutative. 
Namely we have the equality $d_T^ud^{t \ast}_{T\ssm\{t\}}=d_{T\ssm\{u\}}^{t \ast}d^u_{T\ssm\{t\}}$. 
This equality means that 
the definition of $d_T^{t \ast}$ 
does not depend on the choice of $s$ and 
the pair of families 
$(\{a_s\}_{s\in S},\{d_T^{t \ast}:x_{T\ssm\{t\}} \to x_T \}_{T\in \cP(S),\ t\in T})$ 
forms an adjugate of $\Fib \fx$. 
\end{proof}

\begin{thm}
\label{thm:main thm}
Let $\cC$ be an $A$-linear abelian category and 
$x$ an $S$-cube in $\cC$ 
which admits a regular adjugate $\fA=(\{a_s\}_{s\in S},\fd^{\ast})$. 
Then\\
$\mathrm{(1)}$ 
The $S$-cube $x$ is monic. 
In particular, 
we can define 
the patching family $\fP_{\fA}=\{x^{\fA,T}\}_{T\in\cP(S)}$ 
associated to the adjugate $\fA$ as 
in Example~\ref{ex:pat fam ass adj}.\\
$\mathrm{(2)}$ 
The $S$-cube $x^{\fA,T}$ is admissible for any subset $T$ of $S$. 
In particular, the $S$-cube $x=x^{\fA,\emptyset}$ is admissible. 
\end{thm}

\begin{proof}[\bf Proof]
$\mathrm{(1)}$ 
We apply Lemma~\ref{lem:cat fund} $\mathrm{(1)}$ 
to pairs of composable morphisms 
$x_T \onto{d_T^t} x_{T\ssm\{t\}} \onto{d_T^{t\ast}} x_T$ and 
$x_{T\ssm\{t\}} \onto{d_T^{t\ast}} x_T \onto{d_T^t} x_{T\ssm\{t\}}$ 
for any subset $T$ of $S$ and any element $t$ of $T$, 
and we notice that the morphisms $d_T^t$ and $d_T^{t\ast}$ are monomorphisms. 
Hence we obtain the desired result. 

\sn
$\mathrm{(2)}$ 
By replacing $x^{\fA,T}$ and $\fA^T$ with $x$ and $\fA$ respectively, 
we may assume 
without loss of generality 
that $T=\emptyset$. 
We check the assertion for $x^{\fA,\emptyset}=x$. 
If $\# S=1$, then $x$ is admissible by assertion $\mathrm{(1)}$. 
We apply Theorem~\ref{thm:dct} to the double $S$-cube $\Pat \fP_{\fA^{\ast}}$. 
Then we obtain the desired result by the equality 
$e_S^{\ast}\Pat\fP_{\fA^{\ast}}={(\widehat{x^{\ast}})}^{\fA^{\ast},S}=\widehat{{(\widehat{x^{\ast}})}^{\ast}}=x$. 
What we need to check is the following conditions:\\
$\mathrm{(A)}$ 
$\mathbf{2}^{\ast}{(\Pat \fP_{\fA^{\ast}})}$ is admissible.\\
$\mathrm{(B)}$ 
All boundary morphisms are monomorphisms.\\
$\mathrm{(C)}$ 
If $\# S\geq 3$, then 
all the faces of $e_T^{\ast}{(\Pat \fP_{\fA^{\ast}})}$ 
are admissible for any proper subset $T$ of $S$.\\
Condition $\mathrm{(B)}$ follows from 
assertion $\mathrm{(1)}$. 
Since we have the equality 
$\mathbf{2}^{\ast}{(\Pat \fP_{\fA^{\ast}})}=\Typ(\fa;x_S)$ and 
$\fa$ is $x_S$-sequence, 
$\mathbf{2}^{\ast}{(\Pat \fP_{\fA^{\ast}})}$ is admissible 
by Lemma~\ref{lem:char of x-seq}. 
Therefore we get the desired result for $\# S=2$. 
For $\# S\geq 3$, 
we need to check that all the faces of 
$e_T^{\ast}\Pat \fP_{\fA^{\ast}}={(\widehat{(x^{\ast})})}^{\fA^{\ast},T}$ 
are admissible for any proper subset $T$ of $S$. 
By replacing $\widehat{(x^{\ast})}$ and 
$\fA^{\ast}$ with $x$ and $\fA$ respectively, 
we shall just check that all the faces of $x^{\fA,T}$ 
are admissible for any proper subset $T$ of $S$. 
We have the equality 
$${(x^{\fA,T})}|_{S\ssm \{k\}}^V=
\begin{cases}
{(x|_{S\ssm\{k\}}^{\{k\}})}^{\fA|_{S\ssm\{k\}}^{\{k\}},
T\ssm\{k\}} & \text{if $V=\{k\}$}\\
{(x|_{S\ssm\{k\}}^{\emptyset})}^{\fA|_{S\ssm\{k\}}^{\{k\}},
T\ssm\{k\}} & \text{if $V=\emptyset$ and $k\notin T$}\\
{(x|_{S\ssm\{k\}}^{\{k\}})}^{\fA|_{S\ssm\{k\}}^{\{k\}},
T\ssm\{k\}} & \text{if $V=\emptyset$ and $k\in T$} 
\end{cases}$$
for any element $k$ of $S$ and any subset $V$ of $\{k\}$ 
by Lemma~\ref{lem:comp of rest and pat} and 
the patching conditions for $\fP_{\fA}$. 
Therefore the $S\ssm\{k\}$-cube ${(x^{\fA,T})}|_{S\ssm \{k\}}^V$ is admissible by induction of the cardinality of $S$. 
\end{proof}

\sn 
The main theorem has the following application.

\begin{cor}
\label{cor:app 1} 
Let $\ff_S=\{f_s\}_{s\in S}$ and $\{g_s\}_{s\in S}$ be 
families of elements in $A$. 
We set $h_s=f_sg_s$ for any $s\in S$ and $\fh_S:=\{h_s\}_{s\in S}$. 
Assume that $\fh_S$ is an $A$-sequence and $f_s$ is not an invertible element in $A$ 
for any $s\in S$. 
Then $\ff_S$ is also an $A$-sequence. 
\end{cor}

\begin{proof}[\bf Proof] 
We put $d_T^{t\ast}=g_t$ for any $T\in\cP(S)$ and $t\in T$ 
and $\fd^{\ast}:=\{d_T^{t\ast}\}_{T\in\cP(S),\ t\in T}$. 
Then a family $(\fh_S,\fd^{\ast})$ is 
a regular adjugate of 
an $S$-cube $\Typ(\ff_S;A)$. 
Therefore $\Typ(\ff_S;A)$ is admissible 
by Theorem~\ref{thm:main thm} 
and a family $\ff_S$ is an $A$-sequence 
by Lemma~\ref{lem:char of x-seq}.
\end{proof}

\sn
We give an explanation about the relationship 
between 
Theorem~\ref{thm:main thm} and theorem of 
Buchsbaum and Eisenbud \cite{BE73}. 

\begin{df}[\bf Fitting ideal]
\label{df:idealminor}
Let $U$ be an $m \times n$ matrix over $A$ 
where $m$, $n$ are positive integers. 
For $t$ in $(\min(m,n)]$ we then denote by $I_t(U)$ the ideal generated by 
the $t$-minors of $U$, that is, 
the determinant of $t\times t$ sub-matrices of $U$.\\ 
For an $A$-module homomorphism 
$\phi:M\to N$ between free $A$-modules of finite rank, 
let us choose a matrix representation $U$ 
with respect to bases of $M$ and $N$. 
One can easily prove that the ideal $I_t(U)$ 
only depends on $\phi$. 
We put $I_t(\phi):=I_t(U)$ 
and call it the {\bf $t$-th Fitting ideal associated with $\phi$}. 
\end{df}

\begin{nt}[\bf Grade]
\label{nt:grade}
For an ideal $I$ in $A$, we put 
$$S_I:=\{n;\text{There are $f_1,\cdots, f_n\in I$ 
which forms an $A$-regular sequence.}\}, \text{ and}$$
$$\grade I:=
\begin{cases}
0 & \text{if $S_I=\emptyset$}\\
\max S_I & \text{if $S_I$ is a non-empty finite set}\\ 
+\infty & \text{if $S_I$ is an infinite set}.
\end{cases}
$$
\end{nt}

\begin{thm}[\bf Buchsbaum-Eisenbud \cite{BE73}]
\label{thm:BEthm}
Assume that $A$ is noetherian. 
For a complex of free $A$-modules of finite rank.
$$F_{\bullet}:0 \to F_s \onto{\phi_s} F_{s-1} \onto{\phi_{s-1}} \to \cdots \to 
F_1 \onto{\phi_1} F_0 \to 0,$$ 
set $r_i=\overset{s}{\underset{j=i}{\sum}}(-1)^{j-i} \rank F_j$. 
Then the following are equivalent:\\
$\mathrm{(1)}$ $F_{\bullet}$ is $0$-spherical.\\
$\mathrm{(2)}$ $\grade I_{r_i}(\phi_i) \geqq i$ for any $i$ in $(s]$.
\end{thm}

\sn
The following proposition is essentially proven in 
the proof of \cite[4.15]{Kos}.

\begin{prop}
\label{prop:relation with BE}
Let $x$ be an $S$-cube of free $A$-modules of finite ranks. 
Assume that all vertexes of $x$ have the same rank and 
$x$ admits a regular adjugate, 
then $\Tot x$ satisfies the condition $\mathrm{(2)}$ in 
Theorem~\ref{thm:BEthm}. 
\qed
\end{prop}

\end{document}